\documentclass[12pt]{amsart} %or amsproc, or amsbook

\usepackage{amsmath,amssymb,amsthm,amsfonts}
\usepackage[pdftex]{graphicx}
\usepackage{hyperref}

\theoremstyle{plain} %default
\newtheorem{thm}{Theorem}[section]
\newtheorem{lem}[thm]{Lemma}
\newtheorem{prop}[thm]{Proposition}

\theoremstyle{definition}
\newtheorem{defn}{Definition}[section]

\newtheorem{conj}{Conjecture}[section]

\theoremstyle{remark}

\def\Per{{\rm Per}}
\def\Vol{ {\rm Vol} }
\def\RR{\mathbb{R}}

\DeclareMathOperator{\arccot}{arccot}

\begin{document}

\title{Proof of the Log-Convex Density Conjecture}

\author{Gregory R. Chambers}
\begin{abstract}
	We completely characterize isoperimetric regions in
	$\RR^n$ with density $e^h$, where $h$ is convex, smooth, and radially
	symmetric. 
	In particular, balls around the origin constitute isoperimetric regions of any given volume,
	proving the Log-Convex Density Conjecture due to Kenneth Brakke.
\end{abstract}
\maketitle

\tableofcontents

%%%%%%%%%%%%%%%%%%%%%%%%%%%%%%%%%%%%%%%%%%%%%%%%%%%%%%%%
\section{Introduction}
%%%%%%%%%%%%%%%%%%%%%%%%%%%%%%%%%%%%%%%%%%%%%%%%%%%%%%%%

Given a positive function $f$ on $\RR^n$, we define the weighted
perimeter and volume of a set $A \subset \RR^n$ of locally finite perimeter as
	$$\Per(A) = \int_{\partial A} f d \mathcal{H}^{n-1} \qquad {\rm{and}} \qquad \Vol(A) = \int_{A} f d \mathcal{H}^n.$$
Such a positive function $f$ is called a \emph{density} on $\RR^n$.
Here, $\mathcal{H}^m$ indicates the $m$-dimensional Hausdorff measure, and $\partial A$ refers to the essential boundary of $A$.  A good reference for sets of locally finite perimeter and their properties
is the book \cite{maggi_finite} by F. Maggi.  We will use these conventions for the rest of this article, and we also note that all subsets of $\mathbb{R}^n$ involved in this article are of locally finite perimeter.

One natural question immediately arises: for which volumes do isoperimetric regions exists, and can we describe the ones
that do?
Specifically, if we fix a positive weighted volume $M > 0$,
does there exist a set $A \subset \RR^n$ such that
$\Vol(A) = M$ and 
$$\Per(A) = \inf_{\substack{ Q \subset \RR^n \\ \Vol(Q) = M }} \Per(Q).$$

In \cite{morg_iso_1}, Rosales, Ca{\~{n}}ete, Bayle and Morgan consider this problem.  One family of densities that
they examine is radial log-convex densities, that is, densities of the form
	$$ f(x) = e^{g(|x|)} $$
where $g$ is a smooth, convex, and even function on $\RR$.
In particular, for such densities, they show that balls around
the origin are stationary and stable (Corollary 3.11).
By \emph{stable}, we mean that $\Per''(0) \geq 0$
under smooth, volume-conserving variations.  In fact, they show that for any radial, smooth
density $f = e^g$, balls around the origin are stable if and only if $g$ is convex (Theorem 3.10).
This motivates Conjecture \ref*{conj:main} ($3.12$ in their article), first stated by Kenneth Brakke:
\begin{conj}[Log-Convex Density Conjecture]
	\label{conj:main}
	In $\RR^n$ with a smooth, radial, log-convex density, balls around the
	origin provide isoperimetric regions of any given volume.
\end{conj}

	This article seeks to completely solve Conjecture \ref*{conj:main}
by proving the following theorem.
\begin{thm}[Centred Balls are Isoperimetric]
	\label{thm:main}
	Given a density $f(x) = e^{g(|x|)}$ on $\RR^n$ with $g$ smooth, convex and even,
	balls around the origin are isoperimetric regions with respect to weighted perimeter and volume.
\end{thm}
We also prove an additional theorem which morally says that these are the only isoperimetric
regions.  If we have $f$ as in Theorem \ref*{thm:main}, then let
$$ \mathcal{R}(f) = \sup \{ |x| : f(x) = f(0) \}.$$
Since $f$ is convex and radially symmetric, $f = f(0)$ on all of $B_{\mathcal{R}(f)}$.  The purpose of defining
this centered ball is to handle the case of when $f$ is not strictly convex and is constant on some neighborhood
of $0$.
Given this notation, we have that the following uniqueness theorem is true.
\begin{thm}[Uniqueness of Isoperimetric Regions]
	\label{thm:uniqueness}
	Up to sets of measure $0$, the only isoperimetric regions
	are balls centered at the origin, and balls that lie entirely in 
	$$B_{\mathcal{R}(f)} = \{ x : |x| \leq \mathcal{R}(f) \}.$$
\end{thm}

	To prove these statements, we use the observation in \cite{morg_iso_2} that
symmetrization of sets via spherical caps does not increase weighted
perimeter and does not change weighted volume (see the definitions in the next section).  This allows us to reduce our analysis to
the case of such sets.  Furthermore, due to the analysis in \cite{morg_reg_rie}, we may
assume that the boundary of such a set $A$ has a high level of regularity.  Since these sets are volumes of revolution, we can then
use the first variation formula derived in \cite{morg_iso_1} to produce an ODE that is satisfied by each 
curve which, when revolved, produces a component of $\partial A$.  The majority of the article is devoted to analyzing this ODE to prove Theorems
\ref*{thm:main} and \ref*{thm:uniqueness}.  We also note
that the methods here work for all log-convex densities $f(x) = e^{g(| x |)}$ with $g$
in the class $C^3$.  We also note that we do \emph{not} require $g$ to be strictly convex.

A similar approach has also solved another isoperimetric
problem involving $\mathbb{R}^n$ with density.  Specifically, the author along with W. Boyer, W. Brown, A. Loving and S. Tammen proved 
that isoperimetric sets in $\mathbb{R}^n$ with density $f(x) = |x|^p$ are balls whose boundaries pass through the origin.  Here, $p$ can be any positive real number.  This result will appear in an upcoming article.

	This approach of looking at isoperimetric regions that are spherically symmetric and the
associated ODE was used by Kolesnikov and Zhdanov in \cite{kolesnikov} to produce a partial solution
to Conjecture \ref*{conj:main}.  By using these methods in combination with others,
they proved that, if $g$ is strictly convex, then balls of large enough
volume around the origin are isoperimetric regions.
	There have been a number of other partial results in the direction of this conjecture.
For an overview of these, see \cite{fig}.

\bigskip \noindent {\bf Acknowledgments.}
This work was partially supported by 
an NSERC Canadian Graham Bell Graduate Scholarship, and by an Ontario Graduate Scholarship.  The author would like to thank
Almut Burchard, Yevgeny Liokumovich, Alexander Nabutovsky, Sarah Tammen, and Regina Rotman for insightful
conversations related to this problem, and for helpful comments in regard
to the manuscript.
He would also like to thank Frank Morgan for communicating this problem to him, and for many useful comments concerning the initial version of this
article.  Lastly, he would like to thank Frank Morgan, Aldo Pratelli, and Francesco Maggi for observing that the original proof
of the Second Tangent Lemma could be simplified by using the fact that the boundary of an isoperimetric minimizer is regular at a point if it locally lies in a half-space.

%%%%%%%%%%%%%%%%%%%%%%%%%%%%%%%%%%%%%%%%%%%%%%%%%%%%%%%%
\section{Structure of Proof}
%%%%%%%%%%%%%%%%%%%%%%%%%%%%%%%%%%%%%%%%%%%%%%%%%%%%%%%%

In this section, we shall describe the main ingredients of the proofs of Theorems \ref*{thm:main} and \ref*{thm:uniqueness}.
To prove Theorems \ref*{thm:main} and \ref*{thm:uniqueness}, we first note that, without loss of generality,
we may assume that the density function is not constant.
This is because, for a constant density $f = e^c$ on $\RR^n$, we have that
a set $A \subset \RR^n$ of weighted measure $M > 0$ is isoperimetric
if and only if $A$ is isoperimetric in $\RR^n$ with density $1$
among all sets of volume $\frac{M}{{e^c}}$.  This is a classically
solved problem - we know that, up to a set of measure $0$, a set is
isoperimetric if and only if it is a ball.
This proves Theorems
\ref*{thm:main} and \ref*{thm:uniqueness} if $f$ is constant.

For the rest of the article, let us assume that the density $f$ is not constant everywhere.
We see that, since $f$ is log-convex and non-constant, $\lim_{x \rightarrow \infty} f(x) = \infty$.
We can thus use Theorems 3.3 and 5.9 from \cite{morg_iso_2} to show that there exist
isoperimetric regions for each weighted volume $M$, and that the boundary of
each such region is bounded.  We now summarize known properties of
each isoperimetric set $A$.  From \cite{morg_reg_rie}, we have that $\partial A$ is
a smooth $n-1$ dimensional embedded manifold, except on a set of Hausdorff dimension at most $n-8$.
We will call a point $x \in \partial A$ \emph{regular} if $\partial A \cap U$ is an embedded $n-1$ dimensional
manifold, where $U$ is an open subset of $\mathbb{R}^n$ that contains $x$.  By the above comments, the set of points
that are not regular has Hausdorff dimension at most $n-8$.
We also have, by Theorem 6.5 in \cite{morg_iso_2}, that $\partial A$ is
mean curvature convex at all regular points, that is, the mean curvature is positive at
each regular point on $\partial A$.

Lastly, for every $x \in \partial A$, if there is a ball $B$ centered at $x$ such that $B \cap \partial A$
is located in a half-space with respect to an $n-1$ dimensional hyperplane through $x$, then $\partial A$ is regular at $x$.  This is due to the fact
that the oriented tangent cone at $x$ is in a half-space and $A$ is an isoperimetric minimizer, which in turn imply that the oriented tangent cone at $x$
is a hyperplane.  As a result, since the density $f$ is positive everywhere, $\partial A$ is regular at $x$ (see Proposition 3.5 and Remark 3.10 in \cite{morg_reg_rie}).
A corollary of this fact is
that $\partial A$ is regular at each point $x$ with
	$$|x| = \sup_{y \in \partial A } |y|.$$

Next, we analyze the first variation formula at each regular point in the same way as in \cite{morg_iso_2}.
We define a generalized mean curvature $H_f(x)$ at each regular point $x \in \partial A$ by
	$$ H_f(x) = H_0(x) + \frac{\partial h}{\partial \nu}(x)$$
where $H_0$ is the standard inward unaveraged mean curvature, and $\nu$ is the unit outward
normal at $x$.  Here, the function $h$ is defined by $f(x) = e^{h(x)}$.  
If $x \neq 0$, then this can be written as
$$H_f(x) = H_0(x) + g'(| x |) \frac{x}{| x |} \cdot \nu(x),$$
where $g$ is as above, that is, $f(x) = e^{g(|x|)}$.
Henceforth, we shall define $H_1(x) = \frac{\partial g}{\partial \nu}$, so that
$H_f = H_0 + H_1$.
Computing the first variation formula at each regular point $x$ and combining this with the fact
that $A$ is an isoperimetric region, we have that there is
some $c \in \RR$ such that $H_f(x) = c$ at each regular point $x \in \partial A$.

We summarize all of these results in the following theorem:
\begin{thm}[Existence and Regularity of Minimizers]
	\label{thm:regularity}
	For each weighted volume $M$, there exists an isoperimetric region $A$ of weighted
	volume $M$.

	For each such minimizer $A$, we have the following properties:
		\begin{enumerate}
			\item	$\partial A$ is bounded and is a smooth $n-1$ dimensional embedded manifold except on a set of Hausdorff dimension
				at most $n-8$.  Additionally, $\partial A$ is regular at every point at which $\partial A$ locally lies in a half-space, which includes
				all points of maximal magnitude.
			\item	$\partial A$ is mean curvature convex at all regular points.
			\item	For some constant $c \in \RR$, $H_f = c$ at all regular points of $\partial A$.
		\end{enumerate}
\end{thm}

We will first restrict our analysis to a certain type of isoperimetric region.
Let $S_r$ be the $n-1$ dimensional centered sphere of radius $r \geq 0$.    We say
that a minimizer $A$ of weighted volume $M$ is \emph{spherically symmetric}
if, for any $r \geq 0$, $S_r \cap A$
is equal to a closed spherical cap whose center lies on the non-negative $e_1$ axis.  Additionally, if $S_r \cap A$ has $n-1$ dimensional Hausdorff
measure equal to $0$, then $S_r \cap A$ is empty.  Spherically
symmetric isoperimetric regions have all of the properties described in Theorem \ref*{thm:regularity}.
In particular, all irregular points on $\partial A$ lie on the $e_1$ axis (see below), and cannot attain the maximum
magnitude of $\partial A$.

We now work on proving the following theorem, which will aid us in completing the proof.
\begin{prop}
	\label{prop:centered_ball}
	If $f$ is not constant and $A$ is a spherically symmetric isoperimetric region about
	the $e_1$ axis with
		$$\Vol(A) = M$$
	and 
		$$\mathcal{H}^n(A \cap B^c_{\mathcal{R}(f)}) > 0,$$
	then $\partial A = \partial B_M$, where $B_M$ is the centered ball of weighted volume $M$.
\end{prop}
This characteristic will be called the \emph {distributed volume condition}.
We will use the properties of isoperimetric regions described in Theorem \ref*{thm:regularity} to prove this theorem.

We deal first with the case of $n = 1$, which is very straightforward. This is because in this dimension
$H_f(x) = \frac{\partial g} {\partial \nu}(x)$ at every regular point of $\partial A$.
Since $A$ must clearly consist of one interval up to a set of measure $0$, $\partial A$ consists of
the endpoints of this interval.  The fact that $H_f = c$ at each of these endpoints immediately implies that the interval is centered
at the origin, since at least one of the endpoints lies in a region where $g'$ is not 0.

We are left with the case of $n \geq 2$.  To prove Proposition \ref*{prop:centered_ball}, we fix
$n \geq 2$, and choose a spherically symmetric minimizer $A$ of weighted volume $M > 0$.  Since it is spherically symmetric,
there is a closed spherically symmetric set $C \subset \RR^2$  such that, if $\partial C$ is rotated about the $e_1$ axis, we obtain $\partial A$.  We observe that
$\partial C$ is regular at all points that attain the maximum magnitude, and at all points that do not lie on the $e_1$ axis.
Indeed, if there is a point $x \in \partial C$ which is not regular and which does not lie on the $e_1$ axis, then $\partial A$ contains a set of points
of Hausdorff dimension $n - 2$ which are not regular.  Since such a set must have Hausdorff dimension of at most $n - 7$, this is impossible.

Furthermore, the perimeter of $C$ is finite.
All of these properties are directly inherited from $A$.   We would like to
identify a certain continuous curve in $\partial C$.  To identify this curve, we begin by defining $x^*$ to be the
unique point on the positive $e_1$ axis with
	$$|x^*| = \sup_{y \in \partial C} |y|.$$
Since $A$ has positive measure and is spherically symmetric about the non-negative $e_1$ axis, such a point exists and is unique.  Furthermore,
$\partial C$ is locally a $1$-dimensional embedded smooth manifold around $x^*$.  We can use the
spherically symmetric nature of $C$ show that the tangent space of this manifold at $x^*$ is the collection of all multiples
of $e_2$ attached to $x^*$.  As such, this manifold locally only intersects the $e_1$ axis at $x^*$.  Since $\partial C$ is
a smooth embedded manifold at all points that do not lie on the $e_1$ axis, and since it has finite length, we can follow the curve leaving 
$x^*$ in both directions until it intersects the $e_1$ axis at some other point.  This produces a continuous
curve
	$$ \gamma: [-\beta, \beta] \rightarrow \RR^2 $$
with $\beta > 0$.  $\gamma$ has the following properties, where $\gamma = (\gamma_1, \gamma_2)$:

	\begin{enumerate}
		\item $\gamma$ lies on $\partial C$.  
		\item $\gamma(0) = x^*$.
		\item $\gamma(x) \neq \gamma(y)$ unless $x = y$ or $x, y \in \{ \beta, -\beta \}$.  In order words, $\gamma$ forms a simple closed curve.
		\item $\gamma_2 > 0$ on $(0, \beta)$, $\gamma_2 < 0$ on $(-\beta, 0)$, and $\gamma_2 = 0$ at $0$, $-\beta$ and $\beta$.
		\item $\gamma$ is smooth on $(-\beta, \beta)$.
		\item $|\gamma'| = 1$ on $(-\beta, \beta)$.
		\item $\gamma$ is a counterclockwise parametrization.  In particular, since $C$ is
		spherically symmetric,
			$$ | \gamma(x) | \leq | \gamma(y) | $$
			for $x \geq y \geq 0$, and
			$$|\gamma(0)| = |x^*| > \mathcal{R}(f).$$
		\item $\gamma(x) = \gamma(-x)$ for $x \in [0, \beta]$.
	\end{enumerate}
This is shown in Figure \ref*{fig:gamma}.

\begin{figure}[ht]
  \caption{$C$, $\partial C$, $\gamma$, and $\mathcal{K}$.}
  \centering
    \includegraphics[width=1.00\textwidth]{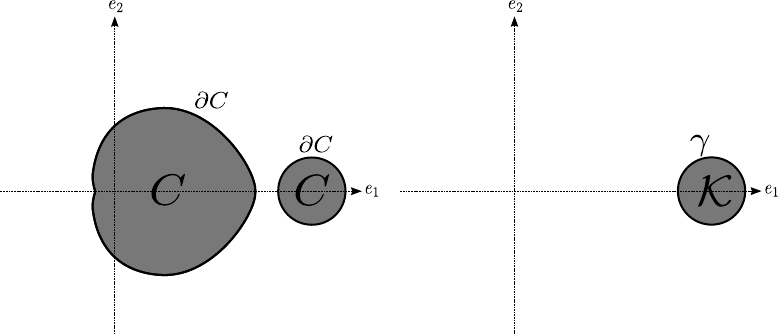}
  \label{fig:gamma}
\end{figure}

Let $\mathcal{K} \subset \RR^2$ be the closed, bounded region enclosed by $\gamma$.  Since $\gamma$ is a simple closed curve, this is well defined.  Due to the fact that $f$ is log-convex,
$f$ is nondecreasing as a function of radius.  Thus, the isoperimetric profile of $\mathbb{R}^n$ with density $f$ is nondecreasing; see Theorem 4.3 in \cite{morg_iso_2}.  In other words, if $\mathcal{J}(V)$
is the minimal perimeter required to enclose volume $V$, then $\mathcal{J}$ is nondecreasing.  This implies that $\mathcal{K} \subset C$, as if it does not, then consider the set $Q \subset \mathbb{R}^n$
defined as the union of $A$ and the result of $\mathcal{K}$ rotated about the $e_1$ axis.  This set will have less perimeter than $A$, but will have larger volume.  Thus, if $V$ is the volume of $A$ and
$V'$ is the volume of $Q$, then $\mathcal{J}(V') < \mathcal{J}(V) = \Per(A)$.  This, however, is a contradiction, since $V' > V$.  The fact that $\mathcal{K} \subset C$ will be useful later in
the argument.

Due to the fact that $C$ is spherically
symmetric, $\mathcal{K}$ is spherically symmetric as well.  When we describe inward and outward normal vectors on $\gamma$, we will be doing
so with respect to $\mathcal{K}$.

We define some notation concerning $\gamma$ on $(-\beta, \beta)$.
	\begin{enumerate}
		\item $H_0(x)$ is the inward unaveraged mean curvature of $\partial A$ at any point that corresponds to $\gamma(x)$, and $H_f(x)$ is the 
			generalized mean curvature of $\partial A$ at any point that corresponds to $\gamma(x)$.  Note that all points corresponding to $x$
			will have the same mean curvature and generalized mean curvature since $A$ is spherically symmetric.
		\item $\kappa(x)$ is the inward curvature of $\gamma$ at $\gamma(x)$.
		\item $n(x)$ is the unit outward normal vector to $\gamma$ at $\gamma(x)$.
		\item $g'(x)$ denotes $ g'( | \gamma(x) | )$.
		\item If $\gamma(x) \neq (0,0)$, then $N(x) = \frac{\gamma(x)}{| \gamma(x) |}$.  This in the unit outward normal
			at $\gamma(x)$ to the centered circle which goes through $\gamma(x)$.
	\end{enumerate}
Note that all of these quantities are defined on $(-\beta, \beta)$.

We will require one more property of $\gamma$, which is control over its tangent vector close to $\beta$.
\begin{lem}
	\label{lem:gamma_behavior_beta}
	If $\lim_{x \rightarrow \beta^-} \gamma'(x)$ exists and is equal to $\nu = (\nu_1, \nu_2)$, then $\nu_1 \leq 0$. 
\end{lem}
\begin{proof}
 	If $\nu_1 > 0$, then $\partial A$ is not regular at the point $x \in \partial A$ that corresponds to $\gamma(\beta)$.	     Furthermore, near $\gamma(\beta)$, the image of $\gamma$ lies to the left of the vertical line through
	$\gamma(\beta)$. Since $\mathcal{K} \subset C$ and $\gamma_1(\beta) < \gamma_1(0)$, this implies that, near $x$,
	$\partial A$ lies in a half-space.  As such, due to Theorem \ref*{thm:regularity}, $\partial A$ is regular
	at $x$.  This is a contradiction, and so $\nu_1 \leq 0$.
\end{proof}

We will complete the proof of Proposition \ref*{prop:centered_ball} by proving
the following two theorems.  Here, we assume the same hypotheses as in Proposition \ref*{prop:centered_ball}, and that $n \geq 2$.

\begin{lem}[First Tangent Lemma]
	\label{lem:first_tangent}
	Either $\gamma$ is a centered circle, or
	there is a point $x \in (0, \beta)$ with $\gamma'(x) = (0,-1)$, and $\kappa(x) > 0$.
\end{lem}

\begin{lem}[Second Tangent Lemma]
	\label{lem:second_tangent}
	If there exists a point $x \in (0, \beta)$ with $\gamma'(x) = (0,-1)$ and $\kappa(x) > 0$, then
	there is another point $y \in (0, \beta)$ with $\gamma'(y) = (0,1)$.
\end{lem}
These are shown in Figure \ref*{fig:lemmas}.

\begin{figure}[ht]
  \caption{Points $x$ and $y$ from Lemmas \ref*{lem:first_tangent} and \ref*{lem:second_tangent}, respectively.}
  \centering
    \includegraphics[width=1.00\textwidth]{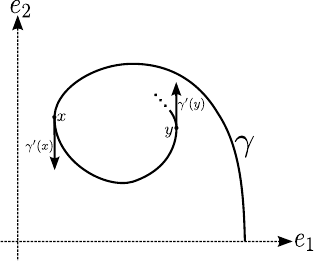}
  \label{fig:lemmas}
\end{figure}

We now observe that such a point $y$ described in Lemma \ref*{lem:second_tangent} cannot exist.  This is 
because of the following lemma, which is a simple result of the fact that $\mathcal{K}$ is spherically symmetric about $e_1$.
In particular, this fact implies that $|\gamma|$ is a non-increasing function on $[0, \beta)$ (see property $6$), and so by differentiating $|\gamma|$,
we obtain the following result. 

\begin{lem}[Tangent Restriction Theorem]
	\label{lem:tangent_restriction}
	For every $x \in (0, \beta)$, $N(x)$ is well-defined and
	$$\gamma'(x) \cdot N(x) \leq 0.$$
\end{lem}

If a point $y$ exists that satisfied the conclusion of Lemma \ref*{lem:second_tangent}, then by Lemma \ref*{lem:tangent_restriction},
	$$\gamma'(y) \cdot N(y) = \frac{\gamma_2(y)}{| \gamma(y) |} \leq 0.$$
Since $\gamma_2(y) > 0$ for each $y \in (0, \beta)$, we thus have a contradiction.  This means that $\gamma$ must be a centered circle of
radius $x^*$.  Thus, $C$ is contained in the closed disc of radius $x^*$, which is equal to $\mathcal{K}$.  Since we have already established that $\mathcal{K} \subset C$, this means
that $C$ is equal to the closed disc of radius $x^*$, and so $A = B_M$, as desired.  This completes the proof of Proposition \ref*{prop:centered_ball}.

With this tool in hand, we can now prove Theorems \ref*{thm:main} and \ref*{thm:uniqueness}.

\begin{proof}[Proof of Theorem \ref*{thm:main} and Theorem \ref*{thm:uniqueness}]
By the logic at the start of this section, this theorem is true if $f$ is constant.
Hence, assume that it is not constant everywhere.  Let $M' = \Vol(B_{\mathcal{R}(f)})$, and fix $M > M' \geq 0$.  By Theorem \ref*{thm:regularity}, there
exists at least one isoperimetric set of weighted volume $M$.
Choose any such minimizer $A$ of weighted volume $M$, and apply the process of \emph{spherical symmetrization} to it. A good reference for this process and its properties is
Section 9.2 of \cite{burago}.
To produce $A^\star$, the result of this process, we consider $S_r$,
the centered $n-1$-dimensional sphere of radius $r$.  We then define
$A^\star \cap S_r$ to be a closed spherical cap
centered on the non-negative $e_1$ axis
such that
$$\mathcal{H}^{n-1}(A^\star \cap S_r) = \mathcal{H}^{n-1}(A \cap S_r).$$
If this measure is $0$, then $A^\star \cap S_r$ is empty.

We then have that, by Theorem 6.2 from \cite{morg_iso_2},
$$\Vol(A^\star) = \Vol(A)$$
and
$$\Per(A^\star) \leq \Per(A).$$
Since $A$ is an isoperimetric region, $A^\star$ is one as well, and so is a
spherically symmetric isoperimetric region.  As $A$ has the distributed volume condition,
$A^\star$ does as well.  Then, by Proposition \ref*{prop:centered_ball},
	$$ \partial A^\star = \partial B_M.$$
Looking again at the definition of $A^\star$, we see that this implies that
	$$ \partial A = \partial B_M,$$
which means that $A$ is a centered ball up to a set of measure $0$.
In summary, if $M > M'$, at least one isoperimetric set $A$ with $\Vol(A) = M$ exists, and for every such
minimizer, it must be a centered ball up to a set of measure $0$.

If $M' = 0$, this completes the proof.  If $M' > 0$, then fix $0 < M \leq M'$. Again, by Theorem \ref*{thm:regularity}, there is at least one isoperimetric set with weighted
volume $M$.  For any such minimizer $A$, if $A$ has the distributed
volume condition then we may apply the above logic to conclude that $A$ is a centered ball, up to a set of measure $0$.  However,
since $A$ has the distributed volume condition, the radius of this ball must be more than $\mathcal{R}(f)$, and so $\Vol(A) > M' \geq M$,
which is a contradiction.
Hence, $A$ lies inside $B_{\mathcal{R}(f)}$, up to a set of measure $0$.  Let us now consider the problem of finding 
isoperimetric regions of weighted volume $M > 0$ in $\RR^n$ with the constant density $h(x) = f(0)$.  We know that the set of such minimizers
is exactly the set of all balls of weighted volume $M$, up to sets of measure $0$.  Thus, $A$ is 
an isoperimetric region if and only if it is a ball of weighted volume $M$ located entirely inside $B_\mathcal{R}(f)$, again up to a set of measure $0$.
\end{proof}

The remainder of this article is divided up into two sections.  The first contains the proof to the First Tangent Lemma,
and the second contains the proof to the Second Tangent Lemma.

%%%%%%%%%%%%%%%%%%%%%%%%%%%%%%%%%%%%%%%%%%%%%%%%%%%
\section{Proof of First Tangent Lemma}
%%%%%%%%%%%%%%%%%%%%%%%%%%%%%%%%%%%%%%%%%%%%%%%%%%%

This section is devoted to proving the First Tangent Lemma (Lemma \ref*{lem:first_tangent}).  We will prove this by contradiction: we will assume that
$\gamma$ is not a centered circle, and then produce a point $p \in (0, \beta)$ with $\gamma'(p) = (0,-1)$ and $\kappa(p) > 0$.
  To go about doing this, we are going to consider two components of the curve $\gamma$, the upper curve, and the lower curve.  Although the definitions
of these two curves are slightly technical, the upper curve will turn out to be the entire first segment of $\gamma$ on 
$[0, \beta)$ such that $\gamma'$ lies in the second
quadrant, and the lower curve will be the entire following segment of $\gamma$ such that $\gamma'$ lies in the third quadrant.  Except for a special case, which
is handled differently, we will prove additional properties of these two curves which allow us to make the following important conclusion.
If we compare a point $x$ on the upper curve and a point $y$ on the lower curve such that
	$$ \gamma_2(x) = \gamma_2(y),$$
then
	$$\kappa(x) \leq \kappa(y).$$
Additionally, $\kappa(x) < \kappa(y)$ for a significant portion of pairs $(x,y)$.  This will allow us to conclude that the lower curve curves faster than the
upper curve, and so must terminate before it reaches the $e_1$-axis.  We will show that the derivative at this point of termination is
$(-1,0)$, as desired.  This is shown in Figure \ref*{fig:upper_lower_curve}.  For the rest of this section, it will be useful to refer back to this figure.

\begin{figure}[ht]
  \caption{The Upper and lower curves of $\gamma$.  Note that $\kappa(x) \leq \kappa(y)$.}
  \centering
    \includegraphics[width=1.00\textwidth]{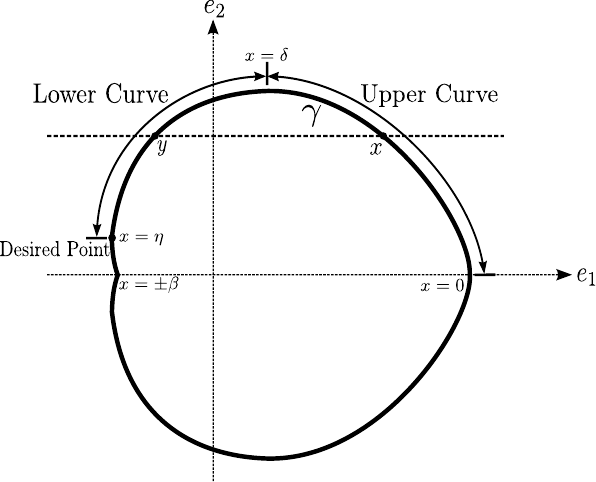}
  \label{fig:upper_lower_curve}
\end{figure}

We will start with several computational results will will be useful to us later.  We will then define the upper curve rigorously, and prove
that it has the desired structure.  Lastly, we will define the lower curve and complete the proof of Lemma \ref*{lem:first_tangent}.

%%%%%%%%%%%%%%%%%%%%%%%%%%%%%%%%%%%%%%%%%%%%%%%%%
\subsection{Preliminary Lemmas and Definitions}
%%%%%%%%%%%%%%%%%%%%%%%%%%%%%%%%%%%%%%%%%%%%%%%%%

We first seek to produce a more manageable expression for the unaveraged
mean curvature $H_0$.  To do this, we define the \emph{canonical circle}:

\begin{defn}[Canonical Circle]
	\label{defn:canonical_circle}
	Given a point $x \in (0,\beta)$, let $C_x$ be the unique oriented circle which
	goes through $\gamma(x)$, whose center lies on the $e_1$ axis, and whose unit tangent
	vector at $\gamma(x)$ is $\gamma'(x)$.  We call this the canonical
	circle at the point $x$, and denote it by $C_x$.  If $\gamma'(x) = (0,1)$ or $\gamma'(x) = (0,-1)$, then
	the canonical circle is an oriented vertical line.    Note that
	this is well-defined since no points in $(0, \beta)$ lie on the $e_1$ axis.  If $x = 0$, then
	we define $C_x$ as $\lim_{x \rightarrow 0^+} C_x$, which exists because $\gamma$ is regular at $0$.
	Lastly, let $\kappa(C_x)$ denote the signed curvature of $C_x$.  If $C_x$ has radius $r$ (possibly $\infty$),
	then $\kappa(C_x) = \frac{1}{r}$ if $C_x$ is counterclockwise oriented, and
	$\kappa(C_x) = - \frac{1}{r}$ if it is clockwise oriented.  From this definition, we see that
	$\kappa(C_0) = \kappa(0)$, the tangent of $C_0$ at $0$ is equal to $\gamma'(0)$, and both the center of
	$C_x$ and the curvature of $C_x$ are continuous functions of $x$ on $[0, \beta)$.

	If the canonical circle at a certain point is an oriented vertical line, then we call the point where it intersects
	the $e_1$ axis its center.
\end{defn}

We can now characterize the unaveraged mean curvature $H_0$ in terms of the inward curvature of $\gamma$
and the signed curvature of the canonical circle.  Note that $\kappa(C_x)$ is a smooth function of $x$ on $(-\beta, \beta)$.

\begin{prop}
	\label{prop:mean_curvature}
	Given a point $x \in [0, \beta)$, we have that
	$$H_0(x) = \kappa(x) + (n-2) \kappa(C_x).$$
\end{prop}
\begin{proof}
	Let us first assume that $\gamma_2(x) > 0$, and that $\gamma'(x) \neq (0, \pm 1)$.
	We can then write $\gamma$ locally around $x$ as a smooth positive function of some portion
	of the $e_1$ axis.  Let this smooth function be $p$, so that $(y_0,p(y_0)) = \gamma(x)$ and $(y, p(y))$
	is equal to $\gamma$ on some local neighborhood of $\gamma(x)$.  Now, we know that every $n$-tuple $(x_1, ..., x_n)$ 
	on the boundary of our spherically symmetric minimizer $A$ in $\RR^n$ that corresponds to a point $(y, p(y))$
	satisfies the following equality:
		$$p(x_1) - \sqrt{\sum_{i=2}^n x_i^2} = 0.$$
	Let $Q(x_1, ..., x_n) = p(x_1) - \sqrt{\sum_{i=2}^n x_i^2}$.  We first compute
		$$div \frac{ \nabla Q}{| \nabla Q |}(x_1, \dots, x_n).$$
	We have that
		$$ \frac{\nabla Q}{| \nabla Q |}(x_1, \dots, x_n) = \frac{1}{\sqrt{1 + (p'(x_1))^2}} (p'(x_1), -\frac{x_2}{\sqrt{\sum_{i=2}^n x_i^2}}, \dots, - \frac{x_n}{\sqrt{\sum_{i=2}^n x_i^2}}).$$
	Taking the divergence of this, we obtain that
	\begin{align*}
		div \frac{\nabla Q}{| \nabla Q |}(x_1, \dots, x_n) &= \frac{p''(x_1)}{\sqrt{1 + (p'(x_1))^2}} - \frac{(p'(x_1))^2 p''(x_1)}{(1 + (p'(x_1))^2)^\frac{3}{2}} \\
						   &  -\frac{1}{\sqrt{1 + (p'(x_1))^2}} \sum_{i = 2}^n (\frac{1}{\sqrt{\sum_{j=2}^n x_j^2}} - \frac{x_i^2}{(\sum_{j=2}^n x_j^2)^{\frac{3}{2}}}) \\
						   &= \frac{p''(x_1)}{(1 + (p'(x_1))^2)^{\frac{3}{2}}} - \\
						   &  (n-2) \frac{1}{\sqrt{(1 + (p'(x_1))^2) (\sum_{i=2}^n x_i^2)}}
	\end{align*}
	Since $p(x_1) = \sqrt{\sum_{i=2}^n x_i^2}$, this becomes
		$$ \frac{p''(x_1)}{(1 + (p'(x_1))^2)^{\frac{3}{2}}} - (n-2) \frac{1}{p(x_1) \sqrt{1 + (p'(x_1))^2}}.$$
	
	Recall that $\mathcal{K}$ is the bounded region enclosed by $\gamma$.  We will break the proof into two cases, depending on
	whether the inward unit normal vector at $x$ with respect to $\mathcal{K}$ is upwards or downwards.  If $n$ is this unit normal vector,
	then we say that it is upwards if $n \cdot (0,1) > 0$, and we say that it is downwards if $n \cdot (0,-1) > 0$.  Since $\gamma' \neq (0, \pm 1)$, these are the only two possibilities.
	 	
	If $n$ is downwards, then from the list of formulae at the end of
	\cite{morg_rie_geom}, we have that the inward mean curvature of $\partial A$
	at $(x_1, \dots, x_n)$ is equal to
		$$- div \frac{ \nabla Q}{| \nabla Q |}(x_1, \dots, x_n) = - \frac{p''(x_1)}{(1 + (p'(x_1))^2)^{\frac{3}{2}}} + (n-2) \frac{1}{p(x_1) \sqrt{1 + (p'(x_1))^2}}.$$
	Since $n$ is downwards, the first term is $\kappa(x)$, and the second term is $(n-2) \kappa(C_x)$, as desired.  In particular, we can show that the second
	term is $(n-2) \kappa(C_x)$ by using a method identical to that employed in Appendix 1, Lemma \ref*{lem:inside_ball_details}.

	Similarly, if $n$ is upwards, then we have that the inward mean curvature of $\partial A$ at $(x_1, \dots, x_n)$ is equal to
		$$ div \frac{\nabla Q}{| \nabla Q |}(x_1, \dots, x_n) = \frac{p''(x_1)}{(1 + (p'(x_1))^2)^{\frac{3}{2}}} - (n-2) \frac{1}{p(x_1) \sqrt{1 + (p'(x_1))^2}}.$$
	In this case, since $n$ is upwards, we have that the first term again corresponds to $\kappa(x)$, and the second term again corresponds to
	$(n-2) \kappa(C_x)$, as desired.  Again, we can show that the second term is $(n-2) \kappa(C_x)$ by using a method identical to that employed in
	Appendix 1, Lemma \ref*{lem:inside_ball_details}.
	
	If $\gamma'(x) = (0, \pm 1)$, then the fact that $x \in [0, \beta)$ implies that
	$\gamma$ is regular at $x$, and so $H_0$, $\kappa$, and $\kappa(C_x)$ are all smooth at $x$.
	Combined with the above result, this smoothness implies that
	$$H_0(x) = \kappa(x) + (n-2) \kappa(C_x)$$
	for all $x$ in $[0, \beta)$.

\end{proof}

We will now prove that $\gamma$ has certain properties which will begin to narrow down
its behavior.  We first state a result that allows us to conclude that $\gamma$
is a circle from local data.
\begin{lem}
	\label{lem:C_x_equality}
	For any point $x \in [0, \beta)$, if we have that the center of $C_x$ is the origin and
	$\kappa(x) = \kappa(C_x)$, then we have that $\gamma$ is a centered circle.
\end{lem}
\begin{proof}
	Since $H_f = c$	for a constant $c$ at all regular points, if $C_x$ is centered on the origin for some $x$ and $\kappa(x) = \kappa(C_x)$,
	we have that $c = g'(x) + (n-1) \kappa( C_x )$.  The two curves $C_x$ and $\gamma$ satisfy the ordinary
	differential equation $H_f = c$, agree at the point $x$, and their tangent vectors
	agree at the point $x$.  By standard theorems concerning the uniqueness of solutions from the theory of ODES, combined with the fact
	that $\gamma$ and $C_x$ are both arclength parametrizations,
	we have that these solutions must locally agree, so $\gamma = C_x$ around $x$.
	By successively applying this local result, we
	have that that $\gamma = C_x$ everywhere.	
\end{proof}

We next characterize the possible centers of $C_x$, for every $x \in [0, \beta)$.
\begin{lem}
	\label{lem:positive_e_1}
	For $x \in [0, \beta)$, if $\gamma'(x)$ is in the second quadrant, and if $a$ is the $e_1$-coordinate of
	the center of $C_x$, then $a \geq 0$.
\end{lem}
\begin{proof}
		First consider the case when $x \neq 0$.  We then have that
	$\gamma_2(x) > 0$.  As a result of Lemma \ref*{lem:tangent_restriction},
		$$N(x) \cdot \gamma'(x) \leq 0.$$
	From the definition of $C_x$ along with the fact
	that $\gamma'(x)$ is in the second quadrant, we have that
		$$ \gamma'(x) = \frac{(-\gamma_2(x), \gamma_1(x) - a)}{| (-\gamma_2(x), \gamma_1(x) - a) |}.$$
	Combining this with $N(x) = \frac{\gamma(x)}{| \gamma(x) |}$, we have that
		$$(\gamma_1(x), \gamma_2(x)) \cdot (-\gamma_2(x), \gamma_1(x) - a) \leq 0,$$
	so we have that $-a \gamma_2(x) \leq 0$.  Since $\gamma_2(x) > 0$, $a \geq 0$.

	If $x = 0$, then since $\kappa(C_0) = \kappa(0)$, the tangent of $C_0$ at $0$ agrees with
	$\gamma'(0)$, and $\kappa'(0) = 0$ (since $\gamma$ is symmetric about $0$), $C_0$ approximates $\gamma$ near $0$ up to the fourth order.  Hence, if $a < 0$,
	then there are points in $[0, \beta)$ near $0$ on $\gamma$ that lie outside the
	centered circle of radius $\gamma(0)$.  This is a contradiction, since all points on $\gamma$ must be less than
	or equal to $|\gamma(0)|$ in magnitude.  Hence, $a \geq 0$ in this case as well.
\end{proof}

To improve readability, we shall define $\lambda(x) = \kappa(C_x)$ for all
$x \in (-\beta, \beta)$.  As already mentioned, $\lambda$ is a smooth function of $x$ on $(-\beta, \beta)$.

We now prove several results, the proofs of which all work the same way.  Given a point $x \in [0, \beta)$,
we consider the unique oriented circle $A_x$ that is tangent to $\gamma$ at $x$, and whose signed curvature is equal to
$\kappa(x)$.  From the definition of curvature, $A_x$ approximates $\gamma$ locally up to the third order
on $(0, \beta)$.  Since $\gamma(x) = \gamma(-x)$ on $[0, \beta)$, we have that $\kappa'(0) = 0$, and so
$A_0$ approximates $\gamma$ up to fourth order locally near $0$.  Note that $A_x$ can be an oriented line.

This will help us prove the following properties of $\gamma$ because we will show that the desired quantities
of $\gamma$ at $x$ are unchanged if we replace $\gamma$ with $A_x$.  In particular, let $\alpha$ be a unit-speed
parametrization of $A_x$, with $\alpha(\widetilde{x}) = \gamma(x)$.  We can define quantities for $\alpha$ in the same way
that we do for $\gamma$.  We shall denote the analogous quantity with a tilde, as follows:
\begin{enumerate}
	\item	$\widetilde{\kappa}(y)$ refers to the signed curvature of $\alpha$ at $y$.
	\item	$\widetilde{H_1}(y)$ refers to $\frac{\partial g}{ \partial \nu}$ at $y$, where $\nu$ is the unit
		outward normal to $\alpha$ at $y$.
\end{enumerate}
Both of these quantities are smooth on the entirety of $\alpha$.

We also define the canonical circle $\widetilde{C}_y$ for points $y$ on $A_x$.  If $\alpha_2(y) \neq 0$, then $\widetilde{C}_y$ is defined as the canonical
circle at $y$ on $A_x$ using the same definition as before.
If $\alpha_2(y) = 0$ and $\alpha'(y) = (0, \pm 1)$, then $\widetilde{C}_y$ is defined to be $A_x$.  If $\alpha_2(y) = 0$ and $\alpha'(y) \neq (0, \pm 1)$, then 
$\widetilde{C}_y$ is not defined.  At each point $y$ where $\widetilde{C}_y$ is defined, the canonical circle is defined on a neighborhood of $y$.
	Additionally, we define $\widetilde{\lambda}(y)$ to be the signed curvature of $\widetilde{C}_y$ (defined at each point $y$ where
$\widetilde{C}_y$ exists).  Since $\widetilde{C}_y$ is defined on an open set, $\widetilde{\lambda}(y)$ is also defined on an open set.  Furthermore, $\widetilde{\lambda}(y)$ is smooth 
on this domain of definition.

We produce one more definition before we continue.

\begin{defn}
\label{defn:function_F}
For each $x \in (-\beta, \beta)$, let $F(x)$ be the $e_1$-coordinate of the center of the canonical circle at $x$.  Additionally, if $A_x$ is the circle as defined above, and $y$ is a point
on $A_x$ where $\widetilde{C}_y$ is defined, then let $\widetilde{F}(y)$ be the $e_1$ coordinate of the circle $\widetilde{C}_y$.
\end{defn}

Clearly, $F$ is smooth at all points $x \in (-\beta, \beta)$ with $\gamma'(x) \neq (0, \pm 1)$.  Additionally, $F$ is smooth at $0$.
$\widetilde{F}$ is smooth on an open set containing every point $y$ with the properties that $\widetilde{C}_y$ is defined at $y$ and $\alpha'(y) \neq (0, \pm 1)$.

For each of the proofs, we will summarize how this argument goes, using the above notation.

\begin{lem}
	\label{lem:curvature_bound}
	Given a point $x \in [0, \beta)$, $\kappa'(x) \geq 0$ if the following properties hold:
	\begin{enumerate}
		\item $\gamma'(x)$ is in the second quadrant
		\item $\kappa(x) = \kappa(C_x) > 0$
	\end{enumerate}
	If in addition, $\gamma(x) \not \in B_{\mathcal{R}(f)}$, $\gamma'(x) \neq (0,1)$, and $C_x$ is not centered at the origin,
	then $\kappa'(x) > 0$.
\end{lem}
\begin{proof}
		In this case, we have that $A_x = C_x$, and both have positive radii.  As such, 
	$A_x$ approximates $\gamma$ up to the third order at $x$.  Because of this, we have that
		$$ \lambda'(x) = \widetilde{\lambda}'(\widetilde{x}) = 0. $$

		We can now reduce the problem to determining the sign of $H_1'(x)$, since
	$H_f'(x) = 0$ and $\lambda'(x) = 0$.  Again, because $\gamma$ is approximated up to the third order by $A_x$ near $x$, 
		$$ H_1'(x) = \widetilde{H_1}'(\widetilde{x}).$$
	Computing $\widetilde{H_1}'(\widetilde{x})$ is an easy exercise (see Appendix 1, Lemma \ref*{lem:curvature_bound_details}) - it is always non-positive,
	and is negative if $\gamma(x) \not \in B_{\mathcal{R}(f)}$, $\gamma'(x) \neq (0,1)$, and $C_x$ is not centered at the origin.
	Note that we also use Lemma \ref*{lem:positive_e_1} to show that the center of $C_x$ is greater than $0$.
\end{proof}

\begin{lem}
	\label{lem:second_derivative_kappa}
	If $\gamma$ is not a centered circle, then $\kappa''(0) > 0$.
\end{lem}
\begin{proof}
		In this case, we have that $C_0 = A_0$, and both approximate $\gamma$ up to the fourth order near $0$.  As such, we again
	have that $\lambda''(0) = \widetilde{\lambda}''(\widetilde{0}) = 0$.  Hence, since $H_f''(0) = 0$, we need only show that $H_1''(0) < 0$.  We next observe
	that $C_0$ cannot be a centered circle, as $\kappa(0) = \kappa(C_0)$, and so by Lemma \ref*{lem:C_x_equality}, $\gamma$ would be
	a centered circle.  Since we assumed that this is not the case, $C_0$ is not centered.   Combining this with the fact
	that $C_0$ approximates $\gamma$ up to the \emph{fourth} order, we have that 
		$$H_1''(0) = \widetilde{H_1}''(\widetilde{0}),$$
	which is negative by a straightforward computation (see Appendix 1, Lemma \ref*{lem:second_derivative_kappa_details}).  Note that this uses the fact that $\gamma(0) \not \in
	B_{\mathcal{R}(f)}$.
\end{proof}

\begin{lem}
	\label{lem:inside_ball}
	If $x \in [0, \beta)$, $\kappa(x) > \kappa(C_x) > 0$, $\gamma'(x)$ lies in the third quadrant
	and $x \in B_{\mathcal{R}(f)}$, then $\kappa'(x) \geq 0$, and $\lambda'(x) \leq 0$.
\end{lem}
\begin{proof}
		We again consider $A_x$, which approximates $\gamma$ at $x$ up to the third order.  Let the center of
	$A_x$ be $(a,b)$.  Since $\kappa(x) > \kappa(C_x) > 0$, $b > 0$.  Due to the assumption that $x \in B_{\mathcal{R}(f)}$,
	$H_1'(x) = 0$, and so $H_0'(x) = 0$ as well.  As such, if we can show that $\lambda'(x) \leq 0$, then we will be done.  To do this,
	we observe that
		$$ \lambda'(x) = \widetilde{\lambda}'(\widetilde{x}) \leq 0.$$
	This can be computed in a straightforward manner (see Appendix 1, Lemma \ref*{lem:inside_ball_details}), using the fact that $\gamma'(x)$ lies in the third quadrant and
	$\kappa(C_x) > 0$.
\end{proof}

\begin{lem}
	\label{lem:outside_ball}
	If $x \in [0, \beta)$ with $\gamma'(x) = (-1,0)$, $\gamma_1(x) > 0$ and $\kappa(x) \geq \kappa(C_x) > 0$, and if
	$\gamma(x) \not \in B_{\mathcal{R}(f)}$, then $\kappa'(x) > 0$.
\end{lem}
\begin{proof}
		We have that $A_x$ approximates $\gamma$ at $x$ up to the third order.  As such,
	$\lambda'(x) = \widetilde{\lambda}'(\widetilde{x})$.  Since $\gamma'(x) = (-1,0)$, $\widetilde{\lambda}'(\widetilde{x}) = 0$, and so
	$\lambda'(x) = 0$ as well.  Since $H_f'(x) = 0$, if we can show that $H_1'(x) < 0$, then
	this will imply that $\kappa'(x) > 0$.  To facilitate this, we use the fact that
		$$H_1'(x) = \widetilde{H_1}'(\widetilde{x}) < 0.$$
	The last fact can be computed directly using the premises of the lemma, especially the fact that
	$\gamma(x) \not \in B_{\mathcal{R}(f)}$ (see Appendix 1, Lemma \ref*{lem:curvature_bound_details}).

\end{proof}

The last portion of this section is devoted to proving two theorems which will be useful tools later.
The first allows us to translate statements about the curvature of the graphs of two functions into comparison
statements regarding their values and derivatives, and the second will allow us to compare values of $H_1 = \frac{\partial g} {\partial \nu}$ for different unit vectors $\nu$.

Let us consider a $C^2$ function $h$ with $h:(a,b) \rightarrow \RR_{\geq 0}$, $b > a$.   Let $t_{h} (x)$ denote the unit tangent vector
$\frac{(1,h'(x))}{|(1, h'(x))|}$, $\theta: S^1 \rightarrow (- \pi, \pi]$ denote the counterclockwise angle from $0$ of a unit vector $v$,
and lastly let $\kappa_{h} (x)$ denote the upward curvature of the graph of $h$ at $x$.  

\begin{prop}[Curvature Comparison Theorem]
	\label{prop:curvature_comparison}
	Consider two $C^2$ functions $f,g:(a,b) \rightarrow \RR_{\geq 0}$ with $b > a$.  Let us make the following assumptions
	on $f$ and $g$ (see Figure \ref*{fig:curvature_comparison}):
	\begin{enumerate}
		\item $\lim_{x \rightarrow b^-} t_f (x)$ and $\lim_{x \rightarrow b^-} t_g (x)$ exist
		\item $\lim_{x \rightarrow b^-} f(x)$ and $\lim_{x \rightarrow b^-} g(x)$ exist
		\item $f'(x) \geq 0$ and $g'(x) \geq 0$ on $(a,b)$
		\item $\lim_{x \rightarrow b^-} f(x) \leq \lim_{x \rightarrow b^-} g(x)$, 
			$\lim_{x \rightarrow b^-} \theta(t_f(x)) \geq \lim_{x \rightarrow b^-} \theta(t_g (x))$, and
			$\kappa_f (x) \leq \kappa_g (x)$ for all $x \in (a,b)$
	\end{enumerate}
	Then, for every $x \in (a,b)$,
		$$ f(x) \leq g(x) $$
	and
		$$ \theta( t_f(x) ) \geq  \theta( t_g(x) ). $$
	Additionally, if there exists a point $y \in (a,b)$ such that
		$$\kappa_f(y) < \kappa_g(y),$$
	then there is some $\phi > 0$ such that
		$$ \phi \leq \theta(t_f (x) ) - \theta(t_g (x) ) $$
	for all $x \in (a, y)$.
\end{prop}
\begin{proof}
	At a given point $x \in (a,b)$, let
	$\theta_f(x) = \theta(t_f(x))$ and $\theta_g(x) = \theta(t_g(x))$.  Since $t_f$ and $t_g$ both lie in the first quadrant,
	$\theta_f$ and $\theta_g$ both lie in the interval $[0, \frac{\pi}{2})$.
	Note that $f'(x) = \tan( \theta_f (x) )$ and $g'(x) = \tan( \theta_g (x) )$.  Thus, if we can show that
	$\theta_f(x) \geq \theta_g(x)$ for all $x \in (a,b)$, then the fact that
	$\lim_{x \rightarrow b^-} f(x) \leq \lim_{x \rightarrow b^-} g(x)$ will also imply that $f(x) \leq g(x)$ for all $x \in (a,b)$.
	To this end, we know that
		$$\kappa_f(x) = \frac{f''(x)}{(1 + (f'(x))^2)^{\frac{3}{2}}}$$
	and
		$$\kappa_g(x) = \frac{g''(x)}{(1 + (g'(x))^2)^{\frac{3}{2}}}.$$
	Since $f''(x) = \frac{\theta'_f(x)}{\cos^2(\theta_f(x))}$ and $g''(x) = \frac{\theta'_g(x)}{\cos^2(\theta_g(x))}$, we have
	that these formulae become
		$$\kappa_f(x) = \cos(\theta_f(x)) \theta'_f(x) = (\sin(\theta_f(x)))'$$
	and
		$$\kappa_g(x) = \cos(\theta_g(x)) \theta'_g(x) = (\sin(\theta_g(x)))'.$$
	Hence, for every $x \in (a,b)$,
		$$ \sin(\theta_f(x)) - \sin(\theta_g(x)) \geq \int_x^b \kappa_g(t) - \kappa_f(t) dt, $$
	since $\lim_{x \rightarrow b^-} \theta_f(x) - \theta_g (x) \geq 0$.
	If $\kappa_f \leq \kappa_g$, then $\sin(\theta_f) \geq \sin(\theta_g)$, and so $\theta_f \geq \theta_g$, as desired.
	
	If there exists a point $y \in (a,b)$ with $\kappa_f(y) < \kappa_g(y)$, then there is some neighborhood of $y$
	on which this is true.  As such, for any $x \in (a,y)$,
		$$\int_x^b \kappa_g(t) - \kappa_f(t) dt \geq c > 0,$$
	implying that $\sin(\theta_f(x)) \geq \sin(\theta_g(x)) + c$ for all such $x$.
	This immediately yields the existence of
	a $\phi > 0$ such that $\theta_f - \theta_g > \phi$ on $(a, y)$.
\end{proof}

\begin{figure}[ht]
\caption{The functions $f$ and $g$.}
\centering
  \includegraphics[width=1.00\textwidth]{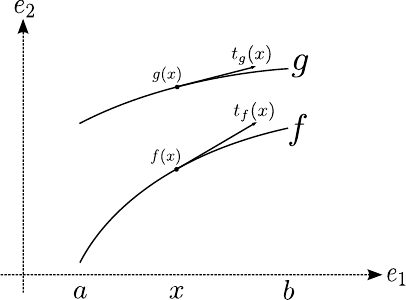}
\label{fig:curvature_comparison}
\end{figure}

We will need a definition before continuing with our second theorem.

\begin{defn}
	\label{defn:admissible}
	Consider a fixed pair of points $(x_1, y)$ and $(x_2, y)$ with
	$y > 0$ and $x_1 \geq x_2$, and a pair of unit vectors
	$v_1, v_2 \in \RR^2$ such that $v_1$ is strictly in the second quadrant and $v_2$ is strictly in the third quadrant.
	Let $C_1$ be the canonical circle
	with respect to $v_1$ at $(x_1,y)$ with center $(a_1,0)$ and radius
	$r_1$, and $C_2$ be the canonical circle
	with respect to $v_2$ at $(x_2,y)$ with center $(a_2,0)$ and radius
	$r_2$.  If $a_1 \geq 0$, $r_2 \geq r_1$, and
	$x_1 - a_1 \geq a_1 - x_2$, then we say that $(v_1, v_2)$ are admissible with respect to
	$(x_1, y)$ and $(x_2, y)$.  Furthermore, let $N(x_1,y) = \frac{(x_1,y)}{|(x_1,y)|}$ and
	define $N(x_2,y)$ similarly.  This is shown in Figure \ref*{fig:admissible}.
\end{defn}

\begin{figure}[ht]
\caption{A pair of vectors $v_1$ and $v_2$ which are admissible with respect to a pair of points $(x_1, y)$ and $(x_2,y)$.}
\centering
  \includegraphics[width=1.00\textwidth]{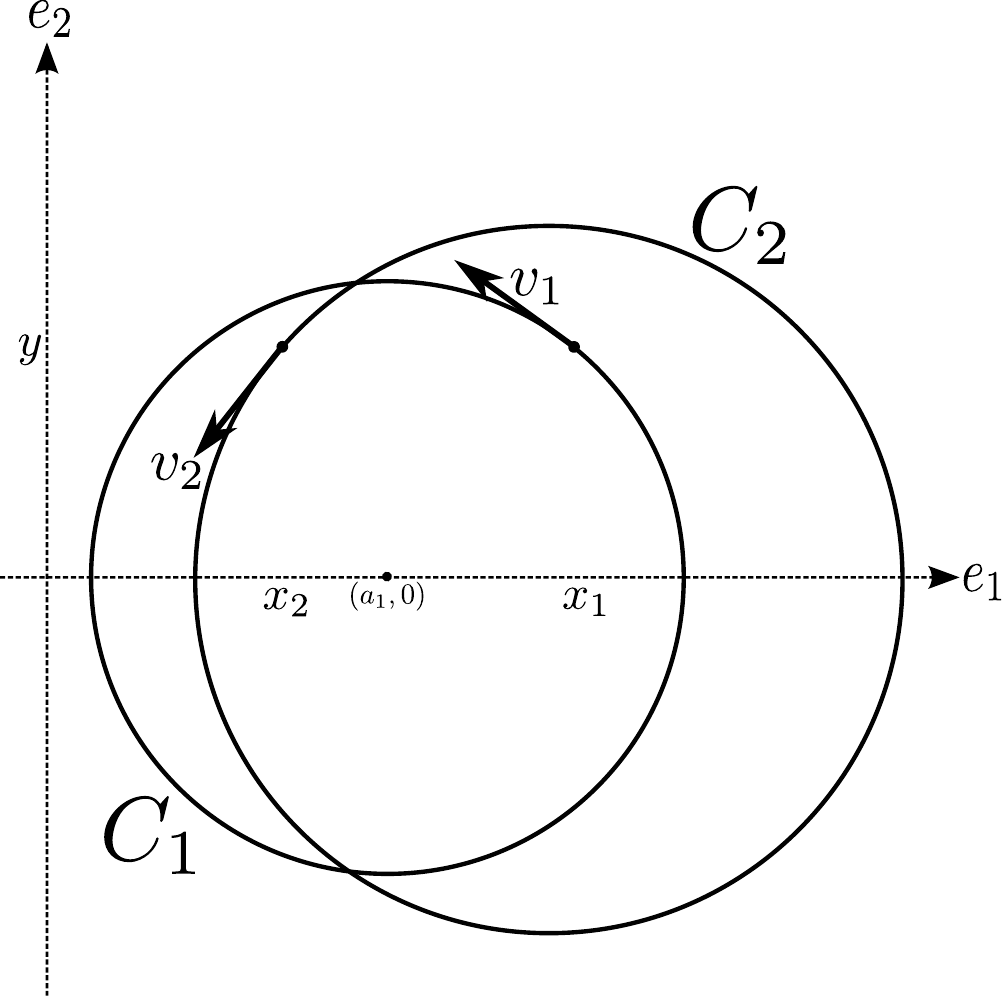}
\label{fig:admissible}
\end{figure}

\begin{prop}[$H_1$ Comparison Theorem]
	\label{prop:normal_computation}
	Consider a pair of points $(x_1,y)$ and $(x_2,y)$ with $y > 0$
	and $x_1 \geq x_2$.  Let $v_1$ and $v_2$ be two unit vectors.
	If $v_1$ and $v_2$ are admissible with respect to $(x_1,y)$ and
	$(x_2,y)$, then
		$$|(x_1,y)| \geq |(x_2,y)|.$$
	Additionally,
		$$v_1^\perp \cdot N(x_1, y) \geq v_2^\perp \cdot N(x_2, y)$$
	with equality if and only if $C_1$ is centered at the origin and
	$C_1 = C_2$.
	Note that here $\perp$ means the perpendicular unit vector formed by a clockwise
	rotation by $\frac{\pi}{2}$ radians.
\end{prop}
\begin{proof}
	To prove that $|(x_1,y)| \geq |(x_2,y)|$, we need only show that
	$|x_1| \geq |x_2|$.  This fact follows immediately from the property that
	$x_1 \geq x_2$, that
	the center of $C_1$ has a non-negative $e_1$-coordinate $a_1$, and that
	$x_1 - a_1 \geq a_1  - x_2$, which are all part of the assumption that
	$v_1$ and $v_2$ are admissible with respect to $(x_1,y)$ and $(x_2,y)$.

	Let $\theta(w)$ denote the counterclockwise angle of the 
	nonzero vector $w$, $\theta: S^1 \rightarrow [0, 2 \pi)$.
	We then have that
		$$v_1^\perp \cdot N(x_1,y) = \cos(\theta(v_1^\perp) - \theta(N(x_1,y)))$$
	and
		$$v_2^\perp \cdot N(x_2,y) = \cos(\theta(v_2^\perp) - \theta(N(x_2,y))).$$

	Let $\theta_1 = \theta(v_1^\perp) - \theta(N(x_1,y))$ and $\theta_2 = \theta(v_2^\perp) - \theta(N(x_2,y))$.  We must show that
		$$\cos(\theta_1) \geq \cos(\theta_2)$$
	with equality if and only if $C_1 = C_2$ and $C_1$ is a centered circle.

	We first observe that there is an $x_1^\star$ such that, if $C_1^\star$ is the canonical
	circle with respect to $v_1$ attached to $(x_1^\star, y)$, then $C_1^\star$ is centered.
	Furthermore, let $x_2^\star = - x_1^\star$, and let $v_2^\star$ be $v_1$ reflected through the $x$-axis.  We then have that
	$v_1$ and $v_2^\star$ are an admissible pair with respect to $(x_1^\star,y)$ and $(x_2^\star,y)$.
	We now have that $C_1 = C_2$ and $C_1$ is centered if and only if $v_2 = v_2^\star$, $x_1 = x_1^\star$, and
	$x_2 = x_2^\star$.  If these conditions are met, then clearly $\theta_1 = \theta_2 = 0$, and so we have equality.
	We will now show that, if at least one of these conditions is not met, then $\cos(\theta_1) > \cos(\theta_2)$, completing
	the proof.

	As a result of the hypotheses, we have that there is some $c,d \geq 0$ and
	$0 \leq \phi \leq \frac{\pi}{2}$ such that $\theta(v_2^\perp) = \theta({v_2^\star}^\perp) + \phi$, $x_1 = x_1^\star + c$, and
	$x_2 = x_2^\star + c + d$, as in Figure \ref*{fig:normal_computation}.  All of the requirements on $x_1, x_2$ and $v_2$ are met if and only if $c,d,\phi = 0$.
	Let us also note that, for a vector $(x,y)$ with $y > 0$, $\theta(N(x,y)) = \arccot(\frac{x}{y})$ and so
		$$\frac{\partial \theta(N(x,y))}{\partial x} = - \frac{y}{x^2 + y^2}.$$
	Denote $\frac{\partial \theta(N(x,y))}{\partial x}$ as $\theta'(N(x,y))$.

	We now see that, since $\theta(N(x_1^\star,y)) = \theta(v_1^\perp)$ and $\theta(N(x_2^\star,y)) + \phi = \theta(v_2^\perp)$,
	\begin{align*}
		\theta_1 &= \theta(N(x_1^\star,y)) - \theta(N(x_1,y)) \\
			 &= -\int_{x_1^\star}^{x_1^\star + c} \theta'(N(t,y)) dt \\
			 &= \int_{x_1^\star}^{x_1^\star + c} \frac{y}{t^2 + y^2} dt
	\end{align*}

	and

	\begin{align*}
		\theta_2 &= \phi + \theta(N(x_2^\star,y)) - \theta(N(x_2,y)) \\
			 &= \phi -  \int_{x_2^\star}^{x_2^\star + c} \theta'(N(t,y)) dt - \int_{x_2^\star + c}^{x_2^\star + c + d} \theta'(N(t,y)) dt \\
			 &= \phi + \int_{x_2^\star}^{x_2^\star + c} \frac{y}{t^2 + y^2} dt + \int_{x_2^\star + c}^{x_2^\star + c + d} \frac{y}{t^2 + y^2} dt.
	\end{align*}
	We first note that $0 \leq \theta_1, \theta_2 \leq \pi$, and so if we can show that
	$\theta_1 < \theta_2$ if at least one of
	$c,d$ or $\phi$ is positive, then we will have that $\cos(\theta_1) > \cos(\theta_2)$, the desired result.
		Due to the fact that $v_1$ lies strictly in the second quadrant,
	$x_1^\star > 0$ and $x_2^\star < 0$.  As such, for each $q \in [0,c]$,
		$$\frac{y}{(x_1^\star + q)^2 + y^2} \leq \frac{y}{(x_2^\star + q)^2 + y^2}$$
	with equality only at $q = 0$.  Hence, if $c > 0$, then
		$$ \int_{x_1^\star}^{x_1^\star + c} \frac{y}{t^2 + y^2} dt < \int_{x_2^\star}^{x_2^\star + c} \frac{y}{t^2 + y^2} dt.$$
	Since $\int_{x_2^\star + c}^{x_2^\star + c + d} \frac{y}{t^2 + y^2} dt \geq 0$ with equality if and only if $d = 0$
	and $\phi \geq 0$ with equality if any only if $\phi = 0$, this completes the proof.
\end{proof}

\begin{figure}[ht]
  \caption{$v_1$, $v_2^\star$, $x_1^\star$ and $x_2^\star$ compared with $v_1$, $v_2$, $x_1$ and $x_2$.}
  \centering
    \includegraphics[width=1.00\textwidth]{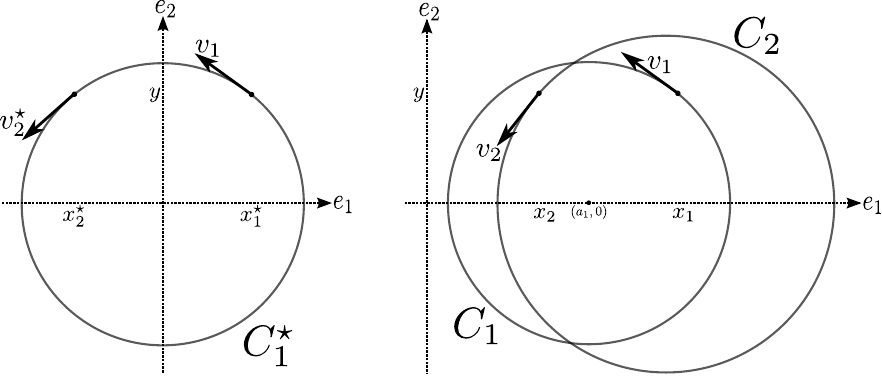}
  \label{fig:normal_computation}
\end{figure}

%%%%%%%%%%%%%%%%%%%%%%%%%%%
\subsection{Upper Curve}
%%%%%%%%%%%%%%%%%%%%%%%%%%%

Here we rigorously define the upper curve, and then use results from the previous section to describe
its behavior.  One should refer to Figure \ref*{fig:upper_lower_curve}.

\begin{defn}
	\label{defn:upper_curve}
	Let the set $K \subset [0, \beta)$ be defined as follows.  A point $x$ is in $K$ if and only if,
	for all $y \in [0, x]$, the following properties are satisfied:
	\begin{enumerate}
		\item $\gamma'(y)$ lies in the second quadrant.
		\item $\kappa(y) \geq \kappa(C_y) > 0$.
		\item $F$ is smooth at $y$ and $F'(y) \geq 0$ (where $F$ is defined as in Definition \ref*{defn:function_F}).
	\end{enumerate}
	Let $\delta = \sup K$.
\end{defn}

We prove a short lemma which will provide sufficient criteria to conclude that $F$ is smooth on a neighborhood of a point, and $F'(x) \geq 0$.

\begin{lem}
	\label{lem:F_and_R}
	If $x \in (0, \beta)$ and $\kappa(x) \geq \kappa(C_x) > 0$,
	then $F$ is smooth at $x$, and $F'(x) \geq 0$.
\end{lem}

\begin{proof}
	As mentioned previously, $F$ exists and is smooth at $x$ since $\gamma_2(x) > 0$, and $\gamma'(x) \neq (0, \pm 1)$ (since $\kappa(C_x) > 0$).
		Using notation from the previous section, $A_x$ locally approximates $\gamma$ at $x$
	up to the third order.  Our first
	observation is that $\widetilde{F}'(\widetilde{x}) = F'(x)$.  This is because $\widetilde{F}$ approximates $F$ near $x$ up to the second order, which is a direct consequence
	of the fact that $A_x$ approximates $\gamma$ at $x$ up to the third order.
		Now we must just show that $\widetilde{F}'(\widetilde{x}) \leq 0$.
	This is the result of a straightforward computation along with the assumptions in the statement of the lemma (see Appendix 1, Lemma \ref*{lem:inside_ball_details}).
\end{proof}

We can now prove the key properties of the upper curve.  We start by demonstrating that $\delta > 0$.
\begin{lem}
	\label{lem:delta_not_0}
	We have that $\delta > 0$.
\end{lem}
\begin{proof}
		To prove that there exists a $\rho > 0$ so that $[0, \rho] \subset K$, we must produce $\rho_1, \rho_2,
	\rho_3 > 0$ so that $\gamma'(y)$ lies in the second quadrant for $y \in [0, \rho_1]$, $\kappa(y) \geq \kappa(C_y) > 0$
	for $y \in [0, \rho_2]$, and $F'(y) \geq 0$ for $y \in [0, \rho_3]$.  We will then
	choose $\rho = \min(\rho_1, \rho_2, \rho_3) > 0$.
		Since $\gamma$ lies inside the centered ball of radius $|\gamma(0)| < \infty$, $\kappa(0) > 0$.  Because $\gamma$
	is smooth on $[0, \beta)$, $\kappa$ is continuous on this interval, and so there is some $\epsilon > 0$ so that
	$\kappa(y) > 0$ on $[0, \epsilon]$.  Furthermore, since $\gamma$ is symmetric about the $e_1$ axis, $\gamma'(0) = (0,1)$.
	Combining these results yields the fact that there is some $\rho_1 > 0$ so that $\gamma'(y)$ is in the second quadrant
	on $[0, \rho_1]$.

	As a result of the fact that $\gamma$ is symmetric about the $e_1$ axis and the definition of the canonical circle, we have that
	$\kappa(0) = \kappa(C_0) > 0$.  Hence, referring to the discussion pertaining to approximating circles in the previous section,
	$A_0 = C_0$ and so $A_0$ approximates $\gamma$ up to the fourth order near $0$ (since $\kappa'(0) = 0$).
	Using the notation $\lambda(x)$ as before, we see that $\lambda(x)$ is constant up to the third order near $0$.  However, by Lemma \ref*{lem:second_derivative_kappa},
	$\kappa''(0) > 0$, and so there is some $\rho_2 > 0$ so that, on $[0,\rho_2]$, $\kappa(x) \geq \kappa(C_x)$.  We may also assume
	that $\kappa(C_x) > 0$ on this interval because $\kappa(C_0) > 0$ and $\kappa$ is continuous on $[0, \beta)$.

		Lastly, we clearly have that $F'(0) = 0$, since $F$ exists and is smooth at $0$, and since $\gamma$ is symmetric about the $e_1$ axis.
	From the previous paragraph, we can find a $\rho_3 > 0$ such that, for $y \in (0, \rho_3]$, $\gamma'(y)$ lies strictly in the second quadrant,
	and $\kappa(y) \geq \kappa(C_y) > 0$.  By Lemma \ref*{lem:F_and_R}, we thus have that $F'(y) \geq 0$ on $(0, \rho_3]$, and so
	this choice of $\rho_3$ is satisfactory.
\end{proof}

We now prove the critical properties of the upper curve.
\begin{prop}[Structure of Upper Curve]
	\label{prop:structure}
	$K$ is not empty, and so $\delta$ exists.  Furthermore, if $\gamma$ is not a centered 
	circle, then $\delta$ has the following properties:
	\begin{enumerate}
		\item $\delta < \beta$
		\item $\delta \in K$
		\item $\gamma_1(\delta) \geq F(x)$ for any $x \in [0, \delta]$.
		\item $\gamma_1(\delta) > 0$
		\item $\gamma'(\delta) = (-1,0)$
	\end{enumerate}
\end{prop}
\begin{proof}
		From Lemma \ref*{lem:delta_not_0}, we have that $\delta$ exists, and $\delta > 0$.  We will begin by proving the first four properties of
	$\delta$ described in the statement of the theorem.  Since $\gamma'(x)$ lies in the second quadrant for
	$x \in [0, \delta)$,
	and $\gamma'(x) \neq (-1,0)$ on a neighborhood of $0$, $\gamma_2(x)$ is a non-decreasing function on $[0, \delta)$ and increases on some interval.
	We thus have that $\lim_{x \rightarrow \delta^-} \gamma_2(x)$ exists and is positive, and so $\delta \neq \beta$.

		Since $\gamma'(0) = (0,1)$, the above paragraph implies that $\gamma$ is smooth on $[0, \delta]$.  Hence, all relevant quantities are continuous on this interval, and so the
	fact that $\delta \in K$ follows directly from this continuity, along with the fact that $\gamma'(\delta) \neq (0,1)$, and $\kappa > 0$ on $K$.
		The property that $F(\delta) \geq F(x)$ for all $x \in [0, \delta]$ is a direct result of the fact that
	$F'(x) \geq 0$ on $[0, \delta]$, and the fact that $F(\delta) \leq \gamma_1(\delta)$ (which is because $\gamma'(\delta)$ lies in the second quadrant).

		We next show that $\gamma_1(\delta) > 0$.  Since $\gamma$ is smooth at $\delta$, $\gamma'(\delta)$ is
	in the second quadrant.  Hence, $\gamma_1(\delta) \geq F(\delta)$.  As a result of the fact that
	$\kappa(0) = \kappa(C_0)$, Lemma \ref*{lem:positive_e_1} and the assumption that $\gamma$ is not a centered
	circle imply that $F(0) > 0$.  Since $F' \geq 0$ on $[0, \delta]$, $F(\delta) > 0$, completing the
	proof.

		The last item that we must prove is that $\gamma'(\delta) = (-1,0)$.  We will prove this by contradiction: if $\gamma'(\delta) \neq (-1,0)$,
	then there is some $\epsilon > 0$ so that $[\delta, \delta + \epsilon] \in K$.  We will break this proof
	down into two cases.  The first case is if $\kappa(\delta) > \kappa(C_\delta) > 0$, and the second case is if $\kappa(\delta) = \kappa(C_x) > 0$.  Since $\delta \in K$,
	we know that these two cases encompass all possibilities.  In both cases, since $\delta > 0$ and $\gamma'(\delta) \neq (-1,0)$, 
	$\gamma'(\delta)$ is strictly in the second quadrant. 
		Let us first assume that $\kappa(\delta) > \kappa(C_\delta) > 0$.  Due to the smoothness of $\gamma$, there exists an $\epsilon > 0$ such that
	$\gamma'$ is in the second quadrant on $[\delta, \delta + \epsilon]$.  Furthermore, since $\lambda(x) = \kappa(C_x)$ and $\kappa$ are both smooth, we may also assume that
	$\kappa(x) > \kappa(C_x) > 0$ on this interval as well.  Lastly, $F'(x) \geq 0$ on this interval as a direct result
	of Lemma \ref*{lem:F_and_R}.  Hence, $[\delta, \delta + \epsilon] \subset K$.

		Let us now assume that $\kappa(\delta) = \kappa(C_\delta) > 0$, and let us consider two sub-cases.  The first sub-case is if
	$\gamma(\delta) \not \in B_{\mathcal{R}(f)}$.  In this sub-case, Lemma \ref*{lem:curvature_bound} tells us that $\kappa'(\delta) > 0$.  We now have that $A_\delta = C_\delta$, and so $C_\delta$ approximates $\gamma$ near $\delta$ up to the third order.
	As such, we see that $\lambda'(\delta) = 0$, and so there exists an $\epsilon > 0$ such that $\kappa(x) \geq \kappa(C_x) > 0$ on
	$[\delta, \delta + \epsilon]$.  Since $\gamma'(\delta)$ is strictly in the second quadrant and $\kappa(x) > 0$ on this interval, we can choose
	$\epsilon$ so that $\gamma'$ is strictly in the second quadrant on $[\delta, \delta + \epsilon]$ as well.  
	The only other fact to check is that we can choose $\epsilon$ so that $F'(x) \geq 0$
on this interval, which is a consequence of
	Lemma \ref*{lem:F_and_R}.

		The remaining sub-case is if $\gamma(\delta) \in B_{\mathcal{R}(f)}$ and $\kappa(\delta) = \kappa(C_\delta) > 0$.  In this case we see that $\gamma$ and
	$C_\delta$ agree at the point $\gamma(\delta)$, and their tangents also agree at this point. As a result, for
	some $\epsilon > 0$, $\gamma = C_\delta$ on $[\delta, \delta + \epsilon]$.  This is because $g'(|x|) = 0$ on $B_{\mathcal{R}(f)}$, and so
	$C_\delta$ and $\gamma$ both satisfy the ODE $\kappa(x) + (n - 2) \kappa(C_x) = \kappa(\delta) + (n - 2) \kappa(C_\delta)$.
	We can apply standard theorems concerning the uniqueness of solutions from the theory of ODEs to
	obtain that, since $C_\delta$ and $\gamma$ are both arclength parametrizations,
	they must be equal on some neighborhood of $\delta$.  As such, we see that $F'(x) = 0$,
	$\kappa(x) = \kappa(\delta) = \kappa(C_\delta) = \kappa(C_x) > 0$, and $\gamma'$ lies in the second quadrant on $[\delta, \delta + \epsilon]$,
	as desired.

\end{proof}

%%%%%%%%%%%%%%%%%%%%%%%%
\subsection{Lower Curve}
%%%%%%%%%%%%%%%%%%%%%%%%

We first will require a definition of the \emph {lower curve}.  We will then compare the lower curve
to the upper curve to prove Lemma \ref*{lem:first_tangent}, as described in the introduction.  Again, it is useful to refer to Figure \ref*{fig:upper_lower_curve} throughout this section.

\begin{defn}
	\label{defn:lower_curve}
	Let $L \subset [\delta, \beta)$ defined as follows.  A point $x \in [\delta, \beta)$ is in $L$ if and only if,
	for every $y \in [\delta, x]$, we have that:
	\begin{enumerate}
		\item $\gamma'(y)$ is in the third quadrant
		\item If $\overline{y}$ is the unique point on the upper curve with $\gamma_2(\overline{y}) = \gamma_2(y)$, then
			$\kappa(\overline{y}) \leq \kappa(y)$.
	\end{enumerate}
	Now define
		$$\eta = \sup L.$$
	Since $\delta \in L$, $\eta$ exists.
\end{defn}

As with the upper curve, we first prove some properties about the lower curve.  To begin, note that
$\gamma_1(x) \leq \gamma_1(\delta)$, and $\gamma_2(x) \leq \gamma_2(\delta)$ for all $x \in L$,
since $\gamma'$ is in the third quadrant on $L$. 
Additionally, observe that, for every $z \in L$, there is a unique $\bar{z}$
in the upper curve such that $\gamma_2(\bar{z}) = \gamma_2(z)$.  This is a result of the fact that $\kappa > 0$ on
the upper curve, $\gamma'$ is in the second quadrant on the upper curve, and $0 \leq \gamma_2(z) \leq \gamma_2(\delta)$.  We make some preliminary conclusions:
since $\kappa(x) > 0$ and $\gamma'(x)$ in the third quadrant for every $x \in L$, if we choose any $z_1, z_2 \in L$,
then $\gamma_2(z_1) \neq \gamma_2(z_2)$.  These comments imply that, for a fixed $x \in L$, we may produce
functions $k$ and $h$ defined on $(\gamma_2(x), \gamma_2(\delta))$ such that $k(t)$ is the unique point $p_1$ in $L$
with $\gamma_2(p_1) = t$, and $h(t)$ is the unique point $p_2$ in $[0, \delta]$ such that
$\gamma_2(p_2) = t$.  Consider now the functions
	$$ Q = 2\gamma_1(\delta) - \gamma_1(h) $$
and
	$$ W = \gamma_1(k) $$
defined on $(\gamma_2(x), \gamma_2(\delta))$.   This is shown in Figure \ref*{fig:Q_and_W}.

\begin{figure}[ht]
  \caption{The functions $Q$ and $W$.}
  \centering
    \includegraphics[width=1.00\textwidth]{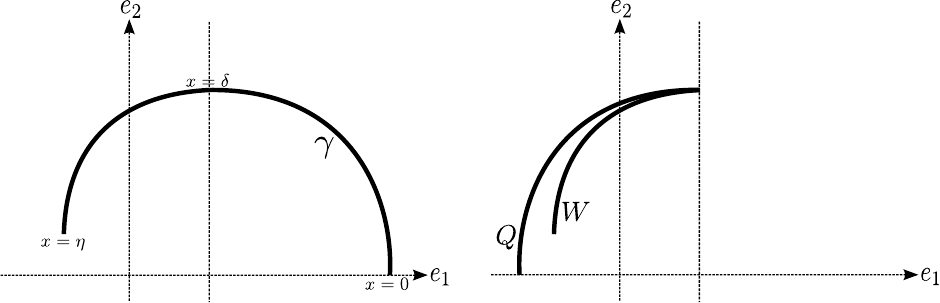}
  \label{fig:Q_and_W}
\end{figure}

We now prove the key properties of the lower curve.  Here again $\theta$ denotes the counterclockwise
angle of a nonzero vector, where $\theta$ takes values in $(-\pi, \pi]$.
\begin{lem}
	\label{lem:lower_curve_properties}
	The following properties of the lower curve are true.  For each $z \in L$,
		\begin{enumerate}
			\item	$\gamma_1(\bar{z}) - \gamma_1(\delta) \geq \gamma_1(\delta) - \gamma_1(z)$
			\item	$\theta( \gamma'(\bar{z})) \geq - \theta(\gamma'(z))$
		\end{enumerate}
\end{lem}
\begin{proof}
	This is a direct consequence of Proposition \ref*{prop:curvature_comparison} applied to $Q$ and $W$.
\end{proof}

We now begin to prove important properties of $\eta$.

\begin{lem}
	\label{lem:lower_curve_eta}
	If $\gamma$ is not a centered circle, then we have that $\eta < \beta$, and $\eta \in L$.
\end{lem}
\begin{proof}
		Consider again the functions $Q$ and $W$.  If $\eta = \beta$, then they are both defined on
	$(0, \gamma_2(\delta))$.  We will show that, for some $\mu \in (0, \gamma_2(\delta))$, $\kappa_W(\mu) > \kappa_Q(\mu)$ (here, we use the upward curvature).
	By the second part of Proposition \ref*{prop:curvature_comparison}, we then have that there is some $\phi > 0$ and $\epsilon > 0$ such that
	$\theta_W + \phi \leq \theta_Q$ at all points in $(0, \epsilon)$.  This translates into the fact that,
	if $z \in L$ with $\gamma_2(z) \in (0, \epsilon)$, then
		$$\theta(\gamma'(\bar{z}) ) \geq \phi - \theta(\gamma'(z)).$$
	Since $\lim_{x \rightarrow 0^+} \theta(\gamma'(x)) = \theta(\gamma'(0)) = \frac{\pi}{2}$, this implies that there is some
	$z \in L$ with $\gamma_2(z) \in (0, \epsilon)$, $\gamma'(z) \neq (-1,0)$, and $-\theta(\gamma'(z)) < \frac{\pi}{2}$.  One of the properties of
	$z \in L$, however, is that $\gamma'(z)$ must lie in the third quadrant, so this yields a contradiction.

		Since $\eta < \beta$, $\gamma$ is smooth at $\eta$.  Hence, $\gamma'(\eta)$ is in the third quadrant, and $\kappa(\eta) \geq \kappa(\overline{\eta})$, and
	so $\eta \in L$.
		We have thus reduced the problem to showing that such a $\mu > 0$ exists.  Since $\gamma(0) \not \in B_{\mathcal{R}(f)}$,
	there is some $0 < \epsilon < \delta$ such that $\gamma(\bar{z}) \not \in B_{\mathcal{R}(f)}$ for all $z \in [\delta, \beta)$ with $\bar{z} \in [0, \epsilon]$.  If we choose any pair
	$(z, \bar{z})$ with $\bar{z} \in [0, \epsilon]$, $\gamma'(z)$ and $\gamma'(\bar{z})$ are admissible vectors with respect to
	$\gamma(z)$ and $\gamma(\bar{z})$, in terms of Definition \ref*{defn:admissible}.
	This is a result of Lemma \ref*{lem:lower_curve_properties} and the fact that
	$F(\bar{z}) \leq F(\delta)$ (Proposition \ref*{prop:structure}).
	Letting $C_1$ and $C_2$ be the two circles as per this definition,
	we cannot have that $C_1 = C_2$ and $C_1$ centered since $| \gamma_1(z) | < | \gamma_1(\bar{z}) |$ as a result of
	the fact that $\gamma_1(\delta) > 0$ (Proposition \ref*{prop:structure}) and $\gamma_1(\delta) - \gamma_1(z) \leq
	\gamma_1(\bar{z}) - \gamma_1(\delta)$ (Lemma \ref*{lem:lower_curve_properties}).  As such, Proposition
	\ref*{prop:normal_computation} yields the conclusion that
		$$| \gamma(z) | \leq | \gamma(\bar{z})  |$$
	and
		$$N(\gamma(z)) \cdot \gamma'(z)^\perp < N(\gamma(\bar{z})) \cdot \gamma'(\bar{z})^\perp.$$
	Combining these results with the facts that $g'(|\gamma(\bar{z})|) > 0,$ and $g'(|\gamma(z)|) \leq g'(|\gamma(\bar{z})|)$, we can conclude that
		$$  g'(|\gamma(z)|) (N(\gamma(z)) \cdot \gamma'(z)^\perp = \frac{\partial g}{\partial \nu}(z) < \frac{\partial g}{\partial \nu}(\bar{z}) = 
		g'(|\gamma(\bar{z})|) N(\gamma(\bar{z})) \cdot \gamma'(\bar{z})^\perp.$$

		Furthermore, from Lemma \ref*{lem:lower_curve_properties}, we have that
		$$\theta(\gamma'(\bar{z})) \geq - \theta(\gamma'(z)),$$
	and so $\kappa(C_z) \leq \kappa(C_{\bar{z}})$.
Using the fact that $H_f(z) = H_f(\bar{z})$, we get that $\kappa(z) < \kappa(\bar{z})$ for all $z$ with $\bar{z} \in (0, \epsilon)$.  This concludes the proof.	
\end{proof}

%%%%%%%%%%%%%%%%%%%%%%%%%%%%%%%%%%%%%%%%%%%%%%%
\subsection{Proof of First Tangent Lemma}
%%%%%%%%%%%%%%%%%%%%%%%%%%%%%%%%%%%%%%%%%%%%%%%

Before completing the proof of Lemma \ref*{lem:first_tangent}, we will need an additional lemma.

\begin{lem}
\label{lem:boundary_points}
If there exists a point $p \in (0, \delta]$ in the upper curve such that
	$$ p \in \partial B_{\mathcal{R}(f)}$$
and
	$$ \kappa(C_p) = \kappa(p), $$
then there exists an $\omega > 0$ such that $p > \omega$, and
	$$ \sup_{y \in [p - \omega, p]} \kappa(y) = \kappa(p).$$
\end{lem}
\begin{proof}
If such an $\omega$ does not exist, then we can find an increasing sequence $a_i \in (0, p)$ with the properties
that
	\begin{enumerate}
		\item	$\sup_{y \in [a_i,p]} \kappa(y) = \kappa(a_i) > \kappa(p)$
		\item 	$\lim_{i \rightarrow \infty} a_i = p.$
	\end{enumerate}

For each $a_i$, choose $b_i \in [a_i,p]$ so that
	$$\kappa(b_i) = \inf_{y \in [a_i,p]} \kappa(y).$$
Of course, $\kappa(b_i) \leq \kappa(p)$ for each $i$.

We next make several observations.  First, we can find an $\epsilon > 0$ such that $p > \epsilon$
and, on $[p - \epsilon, p]$,
	$$H_1' \leq 0.$$
This result can be proved by using the fact that $A_x$ approximates $\gamma$ up to the third order for every regular $x$,
combined with Lemma \ref*{lem:curvature_bound_details} in Appendix 1.  Indeed, since $\kappa(C_p) = \kappa(p)$ and $p \neq 0$,
if we choose $\epsilon$ to be small enough, then for every $x \in [p - \epsilon, p]$, all criteria of Lemma \ref*{lem:curvature_bound_details}
are satisfied, yielding the desired result.  This observation means that, for each $a_i$ and $b_i$ with $i$ sufficiently large,
	$$H_1(a_i) \geq H_1(b_i).$$

The second observation is that there exists a $\xi > 0$ such that, for $a_i$ large enough and
for any $r, s \in [a_i, p]$,
	$$|\kappa(C_r) - \kappa(C_s)| \leq \xi (p - a_i) \cdot (\kappa(a_i) - \kappa(b_i)).$$
Let $D = \kappa(a_i) - \kappa(b_i)$.
The proof of this statement involves comparing $\gamma$ near $p$ to the arc of $C_p$ from $\gamma(p)$ clockwise to the $e_1$-axis.
By analyzing these two curves in a manner similar to that used to prove Proposition \ref*{prop:curvature_comparison}, we have that
the derivatives differ by a constant times $(p - a_i) \cdot D$.  As such, the difference between the curves themselves
is bounded by a constant times $(p - a_i)^2 \cdot D$.  Using the definition of $\kappa(C_x)$, we have that $\kappa(C_r)$ and $\kappa(C_s)$ thus both differ
from $\kappa(C_p)$ by a quantity bounded by a constant times $(p - a_i) \cdot D$, completing the proof.  Note that these estimates make use of the fact that $p - a_i$ and $D$ are sufficiently small,
which we have since $D \rightarrow 0$ as $a_i \rightarrow p^-$.

We now use these two observations to produce a contradiction.  Choose an $a_i$ so large
such that $0 < p - a_i < \frac{1}{2 n \xi}$, $H(b_i) \leq H(a_i)$, and the second observation
above holds for $a_i$.  Using the fact that the generalized mean curvature is constant, we have that
	\begin{align*}
		0	&=	H_f(a_i) - H_f(b_i) \\
			&=	(n-2)(\kappa(C_{a_i}) - \kappa(C_{b_i})) + (\kappa(a_i) - \kappa(b_i)) + ( H_1(a_i) - H_1(b_i)) \\
			&\geq	(n-2)(-\xi (p - a_i) \cdot (\kappa(a_i) - \kappa(b_i))) + \kappa(a_i) - \kappa(b_i)  \\
			&\geq	(1 - (n-2)\xi (p - a_i)) ( \kappa(a_i) - \kappa(b_i) ) \\
			&\geq	\frac{\kappa(a_i) - \kappa(b_i)}{2} \\
			&>	0
	\end{align*}
which is a contradiction, completing the proof.
\end{proof}

We can now prove Lemma \ref*{lem:first_tangent}.  To begin, assume that $\gamma$ is not a centered circle -
 if it is, then we are done.  In particular, Proposition \ref*{prop:structure} and Lemma \ref*{lem:lower_curve_properties} are true.  Also, we will use the fact that
$\eta \in L$ from Lemma \ref*{lem:lower_curve_eta}.  In particular, $\eta \in (0, \beta)$, $\kappa(\eta) > 0$, and $\gamma$ is smooth at $\eta$.  Thus, if $\gamma'(\eta) = (0,-1)$, then 
we will be done.  This is exactly what we will prove.

We will break the rest of the proof into two cases
based on whether $\gamma(\delta) \in B_{\mathcal{R}(f)}$.  We first consider the case when $\gamma(\delta) \not \in B_{\mathcal{R}(f)}$:

\begin{proof}[Proof of First Tangent Lemma - First Case]
In this case,
by Lemma \ref*{lem:outside_ball}, we have that $\kappa'(\delta) > 0$.  This combined with the fact that $\kappa(\delta) > 0$ means that
$\eta > \delta$.  Now let us assume that $\gamma'(\eta) \neq (0,-1)$.
By Lemma \ref*{lem:lower_curve_properties} along with the fact that $F(\bar{\eta}) \leq F(\delta)$
from Proposition \ref*{prop:structure},
we have that the vectors $\gamma'(\eta)$ and $\gamma'(\bar{\eta})$ are admissible with respect to
the points $\gamma(\eta)$ and $\gamma(\bar{\eta})$, as per Definition \ref*{defn:admissible}.  We can now use
Proposition \ref*{prop:normal_computation} to conclude that $\kappa(\eta) > \kappa(\bar{\eta})$, since 
$\gamma( \bar{\eta}) \not \in B_{\mathcal{R}(f)}$, $\gamma_1(\delta) > 0$, and
$\gamma_1(\delta) - \gamma_1(\eta) \leq \gamma_1(\bar{\eta}) - \gamma_1(\delta)$.

Combined with the fact that $\kappa(\eta) > 0$, this means that there is some $\epsilon > 0$ such that
$[\eta, \eta + \epsilon] \subset L$, contradicting the definition of $\eta$.  Hence, $\gamma'(\eta) = (0,-1)$.
\end{proof}

We now turn to the case when $\gamma(\delta) \in B_{\mathcal{R}(f)}$:

\begin{proof}[Proof of First Tangent Lemma - Second Case]
We first reduce the problem to
the case where $\kappa(\delta) = \kappa(C_\delta)$.  If this is not the case, then $\kappa(\delta) > \kappa(C_\delta)$.
Since $\gamma(\delta) \in B_{\mathcal{R}(f)}$, $\gamma(x) \in B_{\mathcal{R}(f)}$ for $x \in [\delta, \beta)$.  Let 
$\mathcal{S}$ be the set of all points $x$ such that, for $y \in [\delta,x]$, $\gamma'(y)$ is in the third quadrant, and
	$$ \kappa(y) > \kappa(C_y) > 0.$$
Consider $\alpha = \sup \mathcal{S}$.  We claim that, for all $x \in [\delta, \alpha)$,
	$$ \kappa(x) > \kappa(\delta) \geq \kappa(C_x) \geq 0.$$
To accomplish this, we use Lemma \ref*{lem:inside_ball}.  This lemma says that $\kappa' \geq 0$ and $\lambda' \leq 0$ on
$[\delta, \alpha)$.  In particular, $\kappa(x) > \kappa(C_\delta)$ for all $x \in [\delta, \alpha)$.  We then
have that $\alpha < \beta$ by comparing
$\gamma$ on $[\delta, \alpha)$ to the portion of $C_\delta$ from $\delta$ to $\beta$ using Proposition \ref*{prop:curvature_comparison}.
In other words, $\gamma$ curves faster that $C_\delta$ on $[\delta, \alpha)$, and so $\alpha < \beta$.   Furthermore, we have that
$\gamma'(\alpha) = (0,-1)$.  If this were not the case, then $\gamma'(\alpha)$ would be in the third quadrant, and
because $\kappa' \geq 0$ and $\lambda' \leq 0$ on $[\delta, \alpha)$,
	$$ \kappa(\alpha) \geq \kappa(\delta) > \kappa(C_\delta) \geq \kappa(C_\alpha) > 0.$$
Hence, $\kappa'(\alpha) \geq 0$ and $\lambda'(\alpha) \leq 0$, and so $\alpha \neq \sup \mathcal{S}$, a contradiction.  Thus,
$\alpha$ satisfies the conclusions of Lemma \ref*{lem:first_tangent}.

	The last case that we must deal with is when $\gamma(\delta) \in B_{\mathcal{R}(f)}$, and $\kappa(\delta) = \kappa(C_\delta)$.  In this case, the two curves and their tangents
agree at $\delta$.
By applying standard theorems concerning the uniqueness of solutions from the theory of ODEs to the ODE $H_f = c$, we see that $\gamma$ must be equal to $C_\delta$ on $\gamma \cap B_{\mathcal{R}(f)}$, since both are arclength parametrizations.  Since
we assumed that a portion of $\gamma$ lies outside $B_{\mathcal{R}(f)}$, there must exist a point $y$ on the upper curve such that
$y \in \partial B_{\mathcal{R}(f)}$.  Additionally, $\gamma'(y)$ is not in the tangent space of $\partial B_{\mathcal{R}(f)}$, since
by Lemma \ref*{lem:C_x_equality}, we would have that $\gamma = B_{\mathcal{R}(f)}$, and so $\gamma$ would be a centered circle, which is assumed to not be the case.
Hence, $\gamma = C_\delta$ on $[y, \beta]$, and $\gamma$ lies outside
$B_{\mathcal{R}(f)}$ on $[0, y)$.  Since $\kappa(C_y) = \kappa(y)$, Lemma \ref*{lem:boundary_points} implies that there is some $\omega > 0$
such that $y > \omega$ and $\kappa \leq \kappa(C_\delta)$ on $[y - \omega, y]$.
	
	Consider the lower curve, which ends at $\eta < \beta$.  We first observe that $\bar{\eta} \leq y - \omega$, since
$\kappa(x) = \kappa(C_\delta)$ for all $x \in [\delta, \eta]$, and so the definition of $\omega$
implies that, if $\bar{\eta} \geq y - \omega$, then $\eta$ cannot be the endpoint of the Lower
Curve (since, in this case, the lower curve can be extended past $\eta$).  We now play the same game as in the first part of this proof.
If $\gamma'(\eta) \neq (0,-1)$, then we can use Proposition \ref*{prop:normal_computation} along with
Lemma \ref*{lem:lower_curve_properties} and Proposition \ref*{prop:structure} to show that
$\kappa(\eta) > \kappa(\bar{\eta})$, which allows us to conclude that $\eta$ is not actually the endpoint of the lower curve, yielding a contradiction.
  
\end{proof}

%%%%%%%%%%%%%%%%%%%%%%%%%%%%%%%%%%%%%%%%%%%%%%%%%%%%
\section{Proof of Second Tangent Lemma}
%%%%%%%%%%%%%%%%%%%%%%%%%%%%%%%%%%%%%%%%%%%%%%%%%%%%

In this section, we will prove Lemma \ref*{lem:second_tangent}.  If there exists a point $x \in [0, \beta)$ with
$\gamma'(x) = (0,-1)$ and $\kappa(x) > 0$, then there exists a point $y \in (0, \beta)$ with $\gamma'(y) = (0,1)$.
Our main tool will be Lemma \ref*{lem:gamma_behavior_beta}; we will show that, if such a point $y$ does not exist, then we obtain a contradiction to this lemma.

The following proposition will constitute the bulk of the work in this section; it investigates the possible behavior of $\gamma$ as $x \rightarrow \beta^-$.
Again, $\theta$ denotes the counterclockwise angle of a unit vector, yielding values in the interval $[0, 2 \pi)$.
\begin{prop}
	\label{prop:continuity_tangent}
	Let $P$ be a point in $(0, \beta)$ with $\gamma'(P) = (0,-1)$, and $\kappa(P) > 0$.
	Then either there exists $\nu \in (P,\beta)$ such that $\gamma'(\nu) = (0,1)$,
	or the following facts are true:
	\begin{enumerate}
		\item	For each $x \in [P, \beta)$, $\kappa(x) > 0$.
		\item	For each $x \in (P, \beta)$, $\gamma'(x)$ strictly lies in the fourth
			quadrant.
		\item	$\lim_{x \rightarrow \beta^-} \gamma'(x)$ exists, lies in the fourth quadrant,
			and is not $(0,-1)$.
	\end{enumerate}
\end{prop}
\begin{proof}
	Consider $\mathcal{S}$, defined as the set of all points $x \in [P, \beta)$ such that, for all
	$y \in [P, x]$, $\gamma'(y)$ lies in the first or fourth quadrant.  Let $\Omega = \sup \mathcal{S}$, and
	let us consider whether $\Omega = \beta$.  If $\Omega \neq \beta$, then we shall show that
	we may chose $\nu = \Omega$, that is, $\gamma'(\Omega) = (0, 1)$ and $\Omega \in (P, \beta)$.

	We first observe that, since $P \in (0,\beta)$ and $\kappa(P) > 0$, $\Omega > P$.  Since
	$\Omega \neq \beta$, we have that $\Omega \in (P, \beta)$.  Assume that $\gamma'(\Omega) \neq (0,1)$.  We then see that the curvature of $C_\Omega$ is less than or equal to 0, and since
	$C$ is mean-convex at all regular points (Theorem \ref*{thm:regularity}), the fact that $\Omega$ is a regular point implies that
	$\kappa(\Omega) > 0$.  There is then some $\epsilon > 0$ such that, for $x \in [\Omega, \Omega + \epsilon]$,
	$x \in \mathcal{S}$.  This contradicts the definition of $\Omega$, and so $\gamma'(\Omega) = (0,1)$, as desired.

	Thus, if $\Omega \neq \beta$, then we are in the first situation.  If $\Omega = \beta$, we shall show that
	all of the properties listed above hold true:
	\begin{enumerate}
		\item	This is a result of the fact that $\gamma$ is mean-curvature convex at every regular point, and the
			curvature of $C_x$ is not positive for all $x \in [P, \beta)$.
		\item	This is a result of the fact that $\kappa > 0$ on $(P, \beta)$.  If there was a point $x \in (P,\beta)$
			with $\gamma'(x)$ not strictly in the fourth quadrant, then this positive curvature property means that
			the derivatives at all points $y \in [x,\beta)$ are in the first quadrant, by the definition of
			$\mathcal{S}$ and the assumption that $\Omega = \beta$.  Hence,
			$\lim_{r \rightarrow \beta^-} \gamma_2(r) \geq \gamma_2(x) > 0$, which is impossible.
		\item	As a consequence of the fact that $\gamma'(x)$ lies strictly in the fourth quadrant for all $x \in (P,\beta)$, and
			$\kappa(x) > 0$ for all such values, $\theta(\gamma'(x))$ is a bounded, increasing function on
			$[P, \beta)$, and so it has a limit in $[\frac{3 \pi}{2}, 2 \pi]$.  Since the angle is strictly
			increasing, the limit cannot be $\frac{3 \pi}{2}$, and so $\lim_{x \rightarrow \beta^-} \gamma'(x)$
			lies in the fourth quadrant and is not $(0,-1)$.
	\end{enumerate}	
\end{proof}

We now prove the Second Tangent Lemma:
\begin{proof}[Proof of Second Tangent Lemma]
	Let $P \in (0,\beta)$ be the given point with $\gamma'(P) = (0,-1)$ and $\kappa(x) > 0$.  Assume that there does not exist a point $y \in (0, \beta)$ such
	that $\gamma'(y) = (0,1)$.  By Proposition \ref*{prop:continuity_tangent}, this implies that $\lim_{x \rightarrow \beta^-} \gamma'(x)$ exists, lies in the fourth quadrant, and is not $(0,-1)$.  If we let $\nu = (\nu_1, \nu_2) = \lim_{x \rightarrow \beta^-} \gamma'(x)$, then this implies that $\nu_1 > 0$.  This contradicts Lemma \ref*{lem:gamma_behavior_beta}, completing the proof.
\end{proof}

%%%%%%%%%%%%%%%%%%%%%%%%%%%%%%%%%%%%%%%%%%%%%%%%%%%%%%%%%%%%%%%%%%%%%%%%%%%%%%%%%%%%%%%%%%%%%%%%%%%%%%%%%%%%%%%%%%%%%%%%%%%%%%%%%%%%
\section{Appendix: Computations}
%%%%%%%%%%%%%%%%%%%%%%%%%%%%%%%%%%%%%%%%%%%%%%%%%%%%%%%%%%%%%%%%%%%%%%%%%%%%%%%%%%%%%%%%%%%%%%%%%%%%%%%%%%%%%%%%%%%%%%%%%%%%%%%%%%%%

In this section, we fill in some computational gaps in the proofs of some of the lemmas in Section 2.
We begin by proving the computational details used in the proof of Lemmas \ref*{lem:curvature_bound} and \ref*{lem:boundary_points}:

\begin{lem}
	\label{lem:curvature_bound_details}
	Consider a circle of radius $r > 0$ centered at point $(a,b)$ with $a \geq 0$ and $b \geq 0$.  Give this curve a counterclockwise orientation,
	and parametrize it by arclength using
		$$ \alpha(x) = (a + r \cos(\frac{x}{r}), b + r \sin(\frac{x}{r})), $$
	where $x \in [0, 2 \pi r)$.
	For any $x \in [0, \frac{\pi r}{2}]$, if $a \sin(\frac{x}{r}) \geq b \cos(\frac{x}{r})$, then
		$$ H_1'(x) \leq 0.$$
	Furthermore, if $\alpha(x) \not \in B_{\mathcal{R}(f)}$ and $a \sin(\frac{x}{r}) > b \cos(\frac{x}{r})$, then
		$$ H_1'(x) < 0.$$
	In particular, if $b = 0$, then the above statements simplify to the following:
	for any $x \in [0, \frac{\pi r}{2}]$,
		$$H_1'(x) \leq 0,$$
	and if $x \in (0, \frac{\pi r}{2}]$, $\alpha(x) \not \in B_{\mathcal{R}(f)}$, and $a > 0$, then
		$$H_1'(x) < 0.$$
\end{lem}
\begin{proof}
	For any $x \in [0, \frac{\pi r}{2}]$, we have that $\gamma(x) \neq (0,0)$.  As such,
		$$ H_1(x) = g'( | \alpha(x) | ) ( N(x) \cdot n(x) ).$$
	Computing, we have that
	\begin{align*}
		(g'( | \alpha | ) N \cdot n )'(x)		&=	g''( | \alpha(x) | ) (\alpha'(x) \cdot \frac{\alpha(x)}{| \alpha(x) |})
											( N(x) \cdot n(x) ) + \\
								&			g'( | \alpha(x) |) ( N'(x) \cdot n(x) ) +
											g'( | \alpha(x) |) ( N(x) \cdot n'(x) ).
	\end{align*}

	We next compute $\alpha'(x) \cdot \alpha(x)$:
	\begin{align*}
		\alpha'(x) \cdot \alpha(x)	&=	(- \sin(\frac{x}{r}), \cos(\frac{x}{r})) \cdot (a + r \cos(\frac{x}{r}),b + r \sin(\frac{x}{r})) \\
						&=	- a \sin(\frac{x}{r}) + b \cos(\frac{x}{r})
	\end{align*}
	Since we assumed that $a \sin(\frac{x}{r}) \geq b \cos(\frac{x}{r})$, $\alpha'(x) \cdot \alpha(x) \leq 0$.  We also know that, due to the convexity
	of $g$, $g''(| \alpha(x) |) \geq 0$.  Furthermore,
	since $\alpha'(x)$ is in the second quadrant and $b \geq 0$, we have that $n(x)$ and $N(x)$ are both in the first quadrant.
	As such, $N(x) \cdot n(x) \geq 0$, and so the first
	term in this expression is less than or equal to $0$.  We also know that $g'(| \alpha(x)|)$ is non-negative, and is positive if $\alpha(x) \not \in B_{\mathcal{R}(f)}$.
	Thus, what remains is to show that
		$$N'(x) \cdot n(x) + N(x) \cdot n'(x) \leq 0$$
	with strict inequality if the conditions in the statement of the lemma are met.

	To this end, we have that
	\begin{align*}
		N(x)		&=	\frac{(a + r \cos(\frac{x}{r}), b + r \sin(\frac{x}{r}))}{\sqrt{(a + r \cos(\frac{x}{r}))^2 + (b + r \sin(\frac{x}{r}))^2}} \\ 
		N'(x)		&=	\frac{(- \sin(\frac{x}{r}), \cos(\frac{x}{r}))}{\sqrt{(a + r \cos(\frac{x}{r}))^2 + (b + r \sin(\frac{x}{r}))^2}} + \\
			        &	(a + r \cos(\frac{x}{r}), b + r \sin(\frac{x}{r}))\frac{ a \sin(\frac{x}{r}) - b \cos(\frac{x}{r}) } { ((a + r \cos(\frac{x}{r}))^2 + (b + r \sin(\frac{x}{r}))^2)^\frac{3}{2}} \\
		n(x)		&=	(\cos(\frac{x}{r}), \sin(\frac{x}{r})) \\
		n'(x)		&=	\frac{( - \sin(\frac{x}{r}), \cos(\frac{x}{r}))}{r} \\	
	\end{align*}

	As such, we have that
	\begin{align*}
		N(x) \cdot n'(x) + N'(x) \cdot n(x)		&=	\frac{- a \sin( \frac{x}{r}) + b \cos( \frac{x}{r})} {r \sqrt{(a + r \cos(\frac{x}{r}))^2 + (b + r \sin(\frac{x}{r}))^2}} + \\
										&	(r + a \cos(\frac{x}{r}) + b \sin(\frac{x}{r})) \cdot \\
										&	\frac{ a \sin(\frac{x}{r}) - b \cos(\frac{x}{r}) } { (\sqrt{(a + r \cos(\frac{x}{r}))^2 + (b + r \sin(\frac{x}{r}))^2})^{\frac{3}{2}}}
	\end{align*}
	This simplifies to
	\begin{align*}
		N(x) \cdot n'(x) + N'(x) \cdot n(x)		&=		(((a + r \cos(\frac{x}{r}))^2 + (b + r \sin(\frac{x}{r}))^2) \cdot \\
										&(- a \sin( \frac{x}{r}) + b \cos( \frac{x}{r}))
											+ r (r + a \cos(\frac{x}{r}) + \\
										& b \sin(\frac{x}{r})) (a \sin(\frac{x}{r}) - b \cos(\frac{x}{r}))) \cdot \\
										&	\frac{1}{r ((a + r \cos(\frac{x}{r}))^2 + (b + r \sin(\frac{x}{r}))^2)^\frac{3}{2}} \\
										&=	(-r (r + a \cos(\frac{x}{r}) \\
										&	- b \sin(\frac{x}{r}))  + ((a + r \cos(\frac{x}{r}))^2 + (b +  \\
										&	r \sin(\frac{x}{r}))^2)) \cdot \\
										&	\frac{-a \sin(\frac{x}{r}) + b \cos(\frac{x}{r})}{r ((a + r \cos(\frac{x}{r}))^2 + (b + r \sin(\frac{x}{r}))^2)^\frac{3}{2}} \\
											\end{align*}

	Rearranging this, we have that
	\begin{align*}
	N(x) \cdot n'(x) + N'(x) \cdot n(x) &=		(ar \cos(\frac{x}{r}) + \\
					    &		3 b r \sin(\frac{x}{r}) + a^2 + b^2) \cdot \\
					    &	\frac{-a \sin(\frac{x}{r}) + b \cos(\frac{x}{r})}{r ((a + r \cos(\frac{x}{r}))^2 + (b + r \sin(\frac{x}{r}))^2)^\frac{3}{2}} \\
	\end{align*}
	
	Since $x \in [0, \frac{\pi r}{2}]$, $\sin(\frac{x}{r})$ and $\cos(\frac{x}{r})$ are both
	non-negative.  Since $a, b \geq 0$, we thus have that all factors in the above expression are non-negative, except for possibly
		$$ -a \sin(\frac{x}{r}) + b \cos(\frac{x}{r}).$$
	Since $a \sin(\frac{x}{r}) \geq b \cos(\frac{x}{r})$, then this term is non-positive, and since $g'(| \alpha(x) |)$ is also non-negative, 
		$$H_1'(x) \leq 0.$$
	
	If $\alpha(x) \not \in B_{\mathcal{R}(f)}$, then $g'( | \alpha(x) |) > 0$.  If, in addition, $a \sin(\frac{x}{r}) > b \cos(\frac{x}{r})$, then
	$-a \sin(\frac{x}{r}) + b \cos(\frac{x}{r}) < 0$.  Additionally, $a > 0$, and so
		$$a r \cos(\frac{x}{r}) + 3 b r \sin(\frac{x}{r}) + a^2 + b^2 > 0.$$
	Hence,
		$$ H_1'(x) < 0.$$
	This completes the proof.

\end{proof}

A remark is in order with regard to the conditions $a \sin(\frac{x}{r}) \geq b \cos(\frac{x}{r})$ and $a \sin(\frac{x}{r}) > b \cos(\frac{x}{r})$.  Let $y \in [0, \beta)$, and suppose that
$\kappa(y) > 0$.  Consider $A_y$, the unique circle which approximates $\gamma$ up to the third order at $\gamma(y)$.  $A_y$ can be parametrized by $\alpha(x) = ( a + r \cos(\frac{x}{r}), b + r \sin(\frac{x}{r}))$,
where $(a,b)$ is the center of $A_y$ and $r > 0$ is its radius.  If $\alpha(x) = \gamma(y)$, then due to Lemma \ref*{lem:tangent_restriction},
	$$ (a + r \cos(\frac{x}{r}), b + r \sin(\frac{x}{r})) \cdot (-\sin(\frac{x}{r}), \cos(\frac{x}{r})) \leq 0, $$
and so
	$$ b \cos(\frac{x}{r}) \leq a \sin(\frac{x}{r}). $$
Hence, for each circle formed in this way, we automatically have one of the conditions of the theorem.  We use additional information about $C_y$ and $\kappa(y)$ to obtain the other conditions.  In the applications of
the form of this lemma where we require that $b \cos(\frac{x}{r}) < a \sin(\frac{x}{r})$, we also use information about $C_y$ and $\kappa(y)$ to conclude that the above inner product has a strict inequality.
 
Next, we compute the details used in the proof of Lemma \ref*{lem:second_derivative_kappa}:
\begin{lem}
	\label{lem:second_derivative_kappa_details}
	Consider a circle $\mathcal{C}$ with radius $r > 0$ and center $(a,0)$ with $a > 0$.  Give this curve a counterclockwise orientation, and parametrize it
	by arclength using
		$$\alpha(x) = (a + r \cos(\frac{x}{r}), r \sin(\frac{x}{r})),$$
	where $x \in [- \pi r, \pi r)$.
	If the point $(a + r , 0)$ is not in $B_{\mathcal{R}(f)}$, we then have that $H_1''(0) < 0$.
\end{lem}
\begin{proof}
		Since $r > 0$, we have that $\alpha(x) \neq (0,0)$ for every $x \in (-\pi r, \pi r)$, and so
			$$H_1(x) = (g'( | \alpha | ) (N \cdot n)) (x).$$
	Hence,	 
		$$H_1''(0) = (g'( | \alpha | ) ( N \cdot n ))''(0).$$

	Computing, we have that 
	\begin{align*}
		(g'( | \alpha | ) ( N \cdot n ) )''(0)	&=	(g''( | \alpha | ) (\alpha' \cdot \frac{\alpha}{| \alpha |}) ( N \cdot n))'(0) + \\
								&	(g'( | \alpha |) (N' \cdot n))'(0) + (g'( | \alpha |) ( N \cdot n' ))'(0).
	\end{align*}
	
	Computing the first term, we have that
	\begin{align*}
		(g''( | \alpha | ) (\alpha' \cdot \frac{\alpha}{|\alpha|})( N \cdot n))'(0)	&=	g'''( | \alpha(0) | ) (\alpha'(0) \cdot \frac{\alpha(0)}{| \alpha(0) |})^2 (N(0) \cdot n(0)) \\
												&	+ g''( | \alpha(0) | ) ((\alpha' \cdot (\frac{\alpha}{| \alpha |}))'(0)) (N(0) \cdot n(0)) \\
												&	+ g''( | \alpha(0) | ) (\alpha'(0) \cdot \frac{\alpha(0)}{| \alpha(0) |}) ((N \cdot n)'(0))
	\end{align*}

	Since $\alpha'(0) \cdot \alpha(0) = 0$, we have that the first and third term are $0$.  Using the fact that $\alpha'(0) \cdot \alpha(0) = 0$, the second term becomes
	\begin{align*}
		g''( | \alpha(0) | ) (N(0) \cdot n(0)) ( \alpha''(0) \cdot \frac{\alpha(0)}{| \alpha(0) |} & \\
		+ \alpha'(0) \cdot (\frac{\alpha'(0)}{| \alpha(0) |} - \frac{\alpha(0)}{| \alpha(0) |^3} (\alpha'(0) \cdot \alpha(0)))) &=
		\frac{g''( | \alpha(0) | ) (N(0) \cdot n(0))}{| \alpha(0) |} \cdot \\
		& ( \alpha''(0) \cdot \alpha(0) + \alpha'(0) \cdot \alpha'(0))
	\end{align*}
	
	We see that
		$$ \alpha'(x) = (-\sin(\frac{x}{r}), \cos(\frac{x}{r})),$$
	and so 
		$$ \alpha''(x) = (-\frac{\cos(\frac{x}{r})}{r}, -\frac{\sin(\frac{x}{r})}{r}).$$
	As such, we then have that the above expression becomes
		$$\frac{g''( | \alpha(0) | ) (N(0) \cdot n(0))}{| \alpha(0) |} (-\frac{a}{r} -1 + 1) \leq 0.$$

	Let us now consider the other terms.  They are
	\begin{align*}
		(g'( | \alpha |) (N' \cdot n))'(0) + (g'( | \alpha |) (N \cdot n'))'(0) &= g''( | \alpha(0) | ) (\alpha'(0) \cdot \frac{\alpha(0)}{| \alpha(0) |}) \\
										    &  (N'(0) \cdot n(0)) + g''( | \alpha(0) | ) \cdot \\
										    &  (\alpha'(0) \cdot \frac{\alpha(0)}{| \alpha(0) |}) (N(0) \cdot n'(0)) + \\
										    &   (g'( | \alpha(0) |) (N' \cdot n + N \cdot n')'(0)
	\end{align*}
	The first two terms are zero because $\alpha'(0) \cdot \alpha(0) = 0$.  Since $\alpha(0) = (a + r, 0) \not \in B_{\mathcal{R}(f)}$, we have that $g'( | \alpha(0) | ) > 0$.  Hence, we must show that
	$(N' \cdot n + N \cdot n')'(0) < 0$.  Since the center of $\mathcal{C}$ is $(a,0)$ with $a > 0$, we can use the method of computation used in the proof of Lemma \ref*{lem:curvature_bound_details} to tell us that
	
		$$N'(x) \cdot n(x) + N(x) \cdot n'(x) = \frac{(-a \sin(\frac{x}{r}))(ar \cos(\frac{x}{r}) + a^2)}{r (a^2 + 2ar \cos(\frac{x}{r}) + r^2)^\frac{3}{2}} $$

	We can now compute $(N' \cdot n + N \cdot n' )'(x)$:
	\begin{align*}
	(N' \cdot n + N \cdot n' )'(x) &=  \frac{a^2 \sin^2(\frac{x}{r}) - a^2 \cos^2(\frac{x}{r}) - \frac{a^3}{r}\cos(\frac{x}{r})}{r (a^2 + 2ar \cos(\frac{x}{r}) + r^2)^\frac{3}{2}}   \\
				       &  + (-a \sin(\frac{x}{r}))(ar \cos(\frac{x}{r}) + a^2) \cdot \\
				       & \frac{3a \sin(\frac{x}{r})}{r  (a^2 + 2ar \cos(\frac{x}{r}) + r^2)^\frac{5}{2}}
	\end{align*}
	At $x = 0$, this just becomes
		$$ - \frac{a^2 + \frac{a^3}{r}}{r (a + r)^3}$$
	which is strictly less than $0$ because $r, a > 0$.  This completes the proof.
\end{proof}

Lastly, we give the computational details used in the proof of Lemmas \ref*{lem:inside_ball} and \ref*{lem:F_and_R}:
\begin{lem}
\label{lem:inside_ball_details}
Let $\mathcal{C}$ be a circle with center $(a,b)$, $b \geq 0$, and radius $r > 0$.  Give $\mathcal{C}$ a counterclockwise orientation, and parametrize it by arclength using
	$$ \alpha(x) = (a + r \cos(\frac{x}{r}), b + r \sin(\frac{x}{r})),$$
where $x \in [0, 2 \pi r)$.
Furthermore, let $\lambda(x)$ be defined in the same way as before: $\lambda(x) = \kappa(C_x)$, where $C_x$ is the canonical circle at $x$.  Here, $C_x$ is defined in the usual way
if $\alpha_2(x) \neq 0$.   If $\alpha_2(x) = 0$ and $\alpha'(x) = (0, \pm 1)$,
then we define $C_x$ as $\mathcal{C}$.  If $\alpha_2(x) = 0$ and $\alpha'(x) \neq (0, \pm 1),$ then $C_x$ is not defined.
Additionally, let us define $F(x)$ as the $e_1$ coordinate of the center of $C_x$, also as before.
We then have that $\lambda$ and $\kappa$ both exist and are smooth on $(0, \pi r)$.  Furthermore, on $[\frac{\pi r}{2}, r \pi)$, $\lambda'(x) \leq 0$, and on $(0, \frac{\pi r}{2}]$, $F'(x) \geq 0$.
\end{lem}
\begin{proof}
	Fix $x \in (0, \pi r)$; we will produce expressions for $\lambda(x)$ and for $F(x)$.  Since $b \geq 0$, $\alpha_2(x) > 0$, and so we know that the line from $\alpha(x)$ through 
the center of $C_x$ is not parallel to the $e_1$ axis.  Hence, we can determine $F(x)$ by computing the point of intersection between these two lines, and we can determine $\lambda(x)$
by computing the length of the line segment connecting $\alpha(x)$ to the center of $C_x$.
	
	At $\alpha(x)$, the tangent vector to $C_x$ is $(- \sin(\frac{x}{r}), \cos(\frac{x}{r}))$, and so the unit inward normal is $(-\cos(\frac{x}{r}), -\sin(\frac{x}{r}))$.  Let 
$r'$ be the radius of $C_x$.  As a result of the above discussion, we have that
	$$ b + r \sin(\frac{x}{r}) - r' \sin(\frac{x}{r}) = 0,$$
and so
	$$ r' = \frac{r \sin(\frac{x}{r}) + b}{\sin(\frac{x}{r})}.$$
As such, 
	$$ \lambda(x) = \frac{\sin(\frac{x}{r})}{r \sin(\frac{x}{r}) + b}. $$
In addition, we clearly have that
	$$ F(x) = a + r \cos(\frac{x}{r}) - r' \cos(\frac{x}{r}),$$
which becomes
	$$ F(x) = a - b \frac{\cos(\frac{x}{r})}{\sin(\frac{x}{r})}.$$
Using these expressions, we see that $F$ and $\lambda$ are both smooth on $(0, \pi r)$.
Differentiating $\lambda$, we obtain
	$$ \lambda'(x) = \frac{b \cos(\frac{x}{r})}{r(r \sin(\frac{x}{r}) + b)^2},$$
which is less than or equal to $0$ on $[\frac{\pi r}{2}, \pi r)$.
Differentiating $F(x)$, we obtain that
	$$ F'(x) = \frac{b}{r} \frac{1}{\sin^2(\frac{x}{r})},$$
which is greater than or equal to $0$ on $(0, \frac{\pi r}{2}]$.
\end{proof}

\bibliographystyle{amsplain}
\bibliography{density_iso_bibliography}

\providecommand{\bysame}{\leavevmode\hbox to3em{\hrulefill}\thinspace}
\providecommand{\MR}{\relax\ifhmode\unskip\space\fi MR }
\providecommand{\MRhref}[2]{%
  \href{http://www.ams.org/mathscinet-getitem?mr=#1}{#2}
}
\providecommand{\href}[2]{#2}
\begin{thebibliography}{1}

\bibitem{morg_iso_1}
V.~Bayle, A.~Ca{\~{n}}ete, F.~Morgan, and C.~Rosales, \emph{On the
  isoperimetric problem in euclidean space with density}, Calculus of
  Variations and Partial Differential Equations \textbf{31} (2008), 27--46.

\bibitem{burago}
Yu.~D. Burago and V.~A. Zalgaller, \emph{Geometric inequalities},
  Springer-Verlag, 1980.

\bibitem{fig}
A.~Figalli and F.~Maggi, \emph{On the isoperimetric problem for radial
  log-convex densities}, Calculus of Variations and Partial Differential
  Equations (2012).

\bibitem{kolesnikov}
A.~Kolesnikov and R.~Zhdanov, \emph{On isoperimetric sets of radially symmetric
  measures}, Contemporary Mathematics, vol. 545, pp.~123--154, American
  Mathematical Society, Providence, RI, 2011.

\bibitem{maggi_finite}
F.~Maggi, \emph{Sets of finite perimeter and geometric variation problems: An
  introduction to geometric measure theory}, 1st ed., Cambridge University
  Press, 2012.

\bibitem{morg_rie_geom}
F.~Morgan, \emph{Riemannian geometry: a beginner's guide}, 2nd ed., A. K.
  Peters, 1998.

\bibitem{morg_reg_rie}
\bysame, \emph{Regularity of isoperimetric hypersurfaces in riemannian
  manifolds}, Transactions of the AMS \textbf{355} (2003).

\bibitem{morg_iso_2}
F.~Morgan and A.~Pratelli, \emph{Existence of isoperimetric regions in
  $\mathbb{R}^n$ with density}, Annals of Global Analysis and Geometry (2012).

\end{thebibliography}

\bigskip

\begin{tabbing}
\hspace*{7.5cm}\=\kill
Gregory R. Chambers                 \\
Department of Mathematics           \\
University of Chicago               \\
Chicago, Illinois, 60637            \\
Canada                              \\
e-mail: chambers@math.uchicago.edu  \\
\end{tabbing}

\end{document}